\documentclass[a4paper,reqno,11pt]{amsart}
\usepackage{amsfonts}
\usepackage{amsmath}
\usepackage[all,cmtip]{xy}
\usepackage[latin2jk]{inputenc}
\usepackage{amssymb}
\usepackage{bbm}
\usepackage{amsthm}
\usepackage{amscd}
\usepackage[left=2.8cm,right=2.8cm,bottom=2.9cm,top=2.9cm]{geometry}
\usepackage{hyperref}
\usepackage{bookmark}

\theoremstyle{definition}
\newtheorem{theorem}{Theorem}[section]
\newtheorem{lemma}[theorem]{Lemma}
\newtheorem{proposition}[theorem]{Proposition}
\newtheorem{definition}[theorem]{Definition}
\newtheorem{corollary}[theorem]{Corollary}

\newtheorem{remark}[theorem]{Remark}
\newtheorem{problem}[theorem]{Problem}

\newtheorem{notation}[theorem]{Notation}

\newtheorem{example}[theorem]{Example}

\begin{document}
\title[$L^{p}$ representations of étale groupoids]{Representations of étale
groupoids on $L^p$-spaces}
\author[Eusebio Gardella]{Eusebio Gardella}
\address{Eusebio Gardella\\
Westfälische Wilhelms-Universität Münster, Fachbereich Mathematik,
Einsteinstraße 62, 48149 Münster, Germany}
\email{gardella@uni-muenster.de}
\urladdr{http://pages.uoregon.edu/gardella/}
\author[Martino Lupini]{Martino Lupini}
\address{Martino Lupini\\
Fakultät für Mathematik, Universität Wien, Oskar-Morgenstern-Platz 1, Room
02.126, 1090 Wien, Austria}
\email{martino.lupini@univie.ac.at}
\urladdr{http://www.lupini.org/}
\curraddr{Mathematics Department\\
California Institute of Technology\\
1200 E. California Blvd\\
MC 253-37\\
Pasadena, CA 91125}
\thanks{Eusebio Gardella was supported by the US National Science Foundation
through Grant DMS-1101742. Martino Lupini was supported by the York
University Elia Scholars Program. This work was completed when the authors
were attending the Thematic Program on Abstract Harmonic Analysis, Banach
and Operator Algebras at the Fields Institute. The hospitality of the Fields
Institute is gratefully acknowledged.}
\dedicatory{}
\subjclass[2000]{Primary 47L10, 22A22; Secondary 46H05}
\keywords{Groupoid, Banach bundle, $L^p$-space, $L^p$-operator algebra,
Cuntz algebra}

\begin{abstract}
For $p\in (1,\infty)$, we study representations of étale groupoids on $L^{p}$%
-spaces. Our main result is a generalization of Renault's disintegration
theorem for representations of étale groupoids on Hilbert spaces. We
establish a correspondence between $L^{p}$-representations of an étale
groupoid $G$, contractive $L^{p}$-representations of $C_{c}(G)$, and tight
regular $L^{p}$-representations of any countable inverse semigroup of open
slices of $G$ that is a basis for the topology of $G$. We define analogs $%
F^{p}(G)$ and $F_{\mathrm{red}}^{p}(G)$ of the full and reduced groupoid
C*-algebras using representations on $L^{p}$-spaces. As a consequence of our
main result, we deduce that every contractive representation of $F^{p}(G)$
or $F_{\mathrm{red}}^{p}(G)$ is automatically completely contractive.
Examples of our construction include the following natural families of
Banach algebras: discrete group $L^{p}$-operator algebras, the analogs of
Cuntz algebras on $L^{p}$-spaces, and the analogs of AF-algebras on $L^{p} $%
-spaces. Our results yield new information about these objects: their
matricially normed structure is uniquely determined. More generally,
groupoid $L^{p}$-operator algebras provide analogs of several families of
classical C*-algebras, such as Cuntz-Krieger C*-algebras, tiling
C*-algebras, and graph C*-algebras.
\end{abstract}

\maketitle
\tableofcontents

\section{Introduction}

Groupoids are a natural generalization of groups, where the operation is no
longer everywhere defined. Succinctly, a groupoid can be defined as a small
category where every arrow is invertible, with the operations being
composition and inversion of arrows. A groupoid is called locally compact
when it is endowed with a (not necessarily Hausdorff) locally compact
topology compatible with the operations; see \cite{paterson_groupoids_1999}.
Any locally compact group is in particular a locally compact groupoid. More
generally, one can associate to a continuous action of a locally compact
group on a locally compact Hausdorff space the corresponding action groupoid
as in \cite{lupini_polish_2014}. This allows one to regard locally compact
groupoids as a generalization of topological dynamical systems.

A particularly important class of locally compact groupoids are those where
the operations are local homeomorphisms. These are the so-called étale---or $%
r$-discrete \cite{renault_groupoid_1980}---groupoids, and constitute the
groupoid analog of actions of discrete groups on locally compact spaces. In
fact, they can be described in terms of partial actions of inverse
semigroups on locally compact spaces; see \cite{exel_inverse_2008}.
Alternatively, one can characterize étale groupoids as the locally compact
groupoids having an open basis of \emph{bisections}, i.e.\ sets where the
source and range maps are injective \cite[Section 3]{exel_inverse_2008}. In
the étale case, the set of all open bisections is an inverse semigroup.

The representation theory of étale groupoids on Hilbert spaces has been
intensively studied since the seminal work of Renault \cite%
{renault_groupoid_1980}; see also the monograph \cite%
{paterson_groupoids_1999}. A representation of an étale groupoid $G$ on a
Hilbert space is an assignment $\gamma \mapsto T_{\gamma }$ of an invertible
isometry $T_{\gamma }$ between Hilbert spaces to any element $\gamma $ of $G$%
. Such an assignment is required to respect the algebraic and measurable
structure of the groupoid. The fundamental result of \cite%
{renault_groupoid_1980} establishes a correspondence between the
representations of an étale groupoid $G$ and the nondegenerate $I$-norm
contractive representations of $C_{c}(G)$. (The $I$-norm on $C_{c}(G)$ is
the analogue of the $L^{1}$-norm for discrete groups. When $G$ is Hausdorff, 
$C_{c}(G)$ is just the space of compactly-supported continuous functions on $%
G$. The non-Hausdorff case is more subtle; see \cite[Definition 3.9]%
{exel_inverse_2008}.) Moreover, such a correspondence is compatible with the
natural notions of equivalence for representations of $G$ and $C_{c}(G)$. In
turn, nondegenerate representations of $C_{c}(G)$ correspond to tight
regular representations of any countable inverse semigroup $\Sigma $ of open
bisections of $G$ that forms a basis for the topology of $G$. Again, such a
correspondence preserves the natural notions of equivalence for
representations of $C_{c}(G)$ and $\Sigma $. Tightness is a nondegeneracy
condition introduced by Exel in \cite[Section 11]{exel_inverse_2008}. In the
case when the set $G^{0}$ of objects of $G$ is compact and zero-dimensional,
one can take $\Sigma $ to be the inverse semigroup of compact open
bisections of $G$. In this case the semilattice $E(\Sigma )$ of idempotent
elements of $\Sigma $ is the Boolean algebra of clopen subsets of $G^{0}$,
and a representation of $G$ is tight if and only if its restriction to $%
E(\Sigma )$ is a Boolean algebra homomorphism.

In this paper, we show how an important chapter in the theory of C*-algebras
admits a natural generalization to algebras of operators on $L^{p}$-spaces.
We prove that the correspondences described in the paragraph above
generalize when one replaces representations on Hilbert spaces with
representations on $L^{p}$-spaces for some Hölder exponent $p$ in $(1,\infty
)$. For $p=2$, one recovers Renault's and Exel's results. Interestingly, the
proofs for $p=2$ and $p\neq 2$ differ drastically. The methods when $p\neq 2$
are based on the characterization of invertible isometries of $L^{p}$-spaces
proved by Banach in \cite[Section 5]{banach_theorie_1993}. The result was
later generalized by Lamperti to not necessarily surjective isometries
between $L^{p}$-spaces \cite{lamperti_isometries_1958}, hence the name
Banach-Lamperti theorem.

Following \cite%
{pisier_completely_1990,merdy_factorization_1996,daws_p-operator_2010} we
say that a representation of a matricially normed algebra $A$ on $%
L^{p}(\lambda )$ is $p$-\emph{completely contractive }if all its
amplifications are contractive when the algebra of $n\times n$ matrices of
bounded linear operators on $L^{p}(\lambda )$ is identified with the algebra
of bounded linear operators on $L^{p}(\lambda \times c_{n})$. (Here and in
the following, $c_{n}$ denotes the counting measure on $n$ points.) If $G$
is an étale groupoid, then the identification between $M_{n}(C_{c}(G))$ and $%
C_{c}(G_{n})$ for a suitable amplification $G_{n}$ of $G$ defines matricial
norms on the algebra $C_{c}(G)$. As a corollary of our analysis a
contractive representation of $C_{c}(G)$ on an $L^{p}$-space is
automatically $p$-completely contractive.

In the case of Hilbert space representations, the universal object
associated to $C_{c}(G) $ is the groupoid C*-algebra $C^{\ast }(G) $, as
defined in \cite[Chapter 3]{paterson_groupoids_1999}. One can also define a
reduced version $C_{\mathrm{red}}^{\ast }(G) $ (see \cite[pages 108-109]%
{paterson_groupoids_1999}), that only considers representations of $C_{c}(G) 
$ that are induced---in the sense of Rieffel \cite[Appendix D]%
{paterson_groupoids_1999}---from a Borel probability measure on the space of
objects of $G$. Amenability of the groupoid $G$ implies that the canonical
surjection from $C^{\ast }(G) $ to $C_{\mathrm{red}}^{\ast }(G) $ is an
isomorphism. In the case when $G$ is a countable discrete group, these
objects are the usual full and reduced group C*-algebras.

A similar construction can be performed for an arbitrary $p$ in $(1,\infty )$%
, and the resulting universal objects are the full and reduced groupoid $%
L^{p}$-operator algebras $F^{p}(G)$ and $F_{\mathrm{red}}^{p}(G)$ of $G$.
When $G$ is a countable discrete group, these are precisely the full and
reduced group $L^{p}$-operator algebras of $G$ as defined in \cite%
{phillips_crossed_2013}; see also \cite{gardella_group_2014}. When $G$ is
the groupoid associated with a Bratteli diagram as in \cite[Section 2.6]%
{renault_c*-algebras_2009}, one obtains the spatial $L^{p}$-analog of an AF
C*-algebra; see \cite{phillips_analogs_2014}. (The $L^{p}$-analogs of UHF
C*-algebras are considered in \cite%
{phillips_simplicity_2013,phillips_isomorphism_2013}.) When $G$ is one of
the Cuntz groupoids defined in \cite[Section 2.5]{renault_c*-algebras_2009},
one obtains the $L^{p}$-analogs of the corresponding Cuntz algebra from \cite%
{phillips_analogs_2012,phillips_simplicity_2013,phillips_isomorphism_2013}. 
\newline
\indent More generally, this construction provides several new examples of $%
L^{p}$-analogs of \textquotedblleft classical\textquotedblright\
C*-algebras, such as Cuntz-Krieger algebras, graph algebras, and tiling
C*-algebras (all of which can be realized as groupoid C*-algebras for a
suitable étale groupoid; see \cite{kumjian_graphs_1997} and \cite%
{paterson_groupoids_1999}). It is worth mentioning here that there seems to
be no known example of a nuclear C*-algebra that cannot be described as the
enveloping C*-algebra of a locally compact groupoid.

The groupoid perspective pursued in this paper contributes to clarify what
are the well-behaved\ representations of algebraic objects---such as the
Leavitt algebras, Bratteli diagrams, or graphs---on $L^{p}$-spaces. In \cite%
{phillips_analogs_2012,phillips_simplicity_2013,phillips_isomorphism_2013},
several characterizations are given for well behaved representations of
Leavitt algebras and stationary Bratteli diagrams. The fundamental property
considered therein is the uniqueness of the norm that they induce. The
groupoid approach shows that these representations are precisely those
coming from representations of the associated groupoid or, equivalently, its
inverse semigroup of open bisections.

Another upshot of the present work is that the groupoid $L^{p}$-operator
algebras $F^{p}(G)$ and $F_{\mathrm{red}}^{p}(G)$ satisfy an automatic $p$%
-complete contractiveness\ property for contractive homomorphisms into other 
$L^{p}$-operator algebras. In fact, $F^{p}(G)$ and $F_{\mathrm{red}}^{p}(G)$
have canonical matrix norms. Such matrix norm structure satisfies the $L^{p}$%
-analog of Ruan's axioms for operator spaces as defined in \cite[Subsection
4.1]{daws_p-operator_2010} building on \cite%
{pisier_completely_1990,merdy_factorization_1996}. Using the terminology of 
\cite[Subsection 4.1]{daws_p-operator_2010}, this turns the algebras $%
F^{p}(G)$ and $F_{\mathrm{red}}^{p}(G)$ into $p$-operator systems such that
the multiplication is $p$-completely contractive. It is a corollary of our
main results that any contractive representation of these algebras on an $%
L^{p}$-space is automatically $p$-completely contractive. As a consequence
the matrix norms on $F^{p}(G)$ and $F_{\mathrm{red}}^{p}(G)$ are uniquely
determined---as it is the case for C*-algebras.

It is still not clear what are the well-behaved algebras of operators on $%
L^{p}$-spaces. Informally speaking, these should be the $L^{p}$-operator
algebras that behave like C*-algebras. The results in this paper provide
strong evidence that $L^{p}$-operator algebras of the form $F^{p}(G)$ and $%
F_{\mathrm{red}}^{p}(G)$ for some étale groupoid $G$, indeed behave like\
C*-algebras. Beside having the automatic complete contractiveness property
for contractive representations on $L^{p}$-spaces, another property that $%
F^{p}(G)$ and $F_{\mathrm{red}}^{p}(G)$ share with C*-algebras is being
generated by spatial partial isometries as defined in \cite%
{phillips_analogs_2012}. These are the partial isometries whose support and
range idempotents are hermitian operators in the sense of \cite%
{lumer_semi-inner-product_1961}; see also \cite{berkson_hermitian_1972}. (In
the C*-algebra case, the hermitian idempotents are precisely the orthogonal
projections.) In particular, this property forces the algebra to be a
C*-algebra in the case $p=2$. (A stronger property holds for unital
C*-algebras, namely being generated by invertible isometries; see \cite[%
Theorem~II.3.2.16]{blackadar_operator_2006}. As observed by Chris Phillips,
this property turns out to fail for some important examples of algebras of
operators on $L^{p}$-spaces, such as the $L^{p}$-analog of the Toeplitz
algebra.)

The present work indicates that the properties of being generated by spatial
partial isometries, and having automatic complete contractiveness for
representations on $L^{p}$-spaces, are very natural requirements for an $%
L^{p}$-operator algebra to behave like a C*-algebra. 
We believe that the results of this paper are a step towards a successful
identification of those properties that characterize the class of
\textquotedblleft well behaved" $L^{p}$-operator algebras.

\subsection{Notation}

\label{Subsection: notation} We denote by $\omega $ the set of natural
numbers including $0$. An element $n\in \omega $ will be identified with the
set $\{0,1,\ldots ,n-1\}$ of its predecessors. (In particular, 0 is
identified with the empty set.) We will therefore write $j\in n$ to mean
that $j$ is a natural number and $j<n$.\newline
\indent For $n\in \omega $ or $n=\omega $, we denote by $c_{n}$ the counting
measure on $n$. We denote by $\mathbb{Q}(i)^{\oplus \omega }$ the set of all
sequences $(\alpha _{n})_{n\in \omega }$ of complex numbers in $\mathbb{Q}%
(i) $ such that $\alpha _{n}=0$ for all but finitely many indices $n\in
\omega $.\newline
\indent All Banach spaces will be \emph{reflexive}, and will be endowed with
a (Schauder)\emph{\ basis}. We recall here some terminology concerning bases
in Banach spaces. For more information and details we refer the reader to
the monographs \cite{carothers_short_2005,megginson_introduction_1998}.
Suppose that $(b_{n})_{n\in \omega }$ is a basis of a Banach space $Z$. We
denote by $(b_{n}^{\prime })_{n\in \omega }$ the associated sequence of 
\emph{coefficient functionals}. If $k\in \omega $, then $b_{k}^{\prime}$ is
the element of the dual space $Z^{\prime }$ of $Z$ that maps $z\in Z$ to the 
$k$-th coefficient of $z$ with respect to the basis $(b_{n})_{n\in \omega }$%
. If $Z$ is reflexive, then $(b_{n}^{\prime })_{n\in \omega }$ is a basis
for $Z^{\prime }$ \cite[Proposition 5.3]{singer_bases_1970}. The basis $%
(b_{n})_{n\in \omega }$ is:

\begin{itemize}
\item \emph{unconditional }if, for every $x\in Z$, the series $%
\sum_{n}b_{n}^{\prime }(x)b_{n}$ converges unconditionally to $x$ (this
means that for any bijection $\pi\colon \omega \rightarrow \omega $ the
series $\sum_{n}b_{\pi (n)}^{\prime }b_{n}$ converges to $x$);

\item \emph{normal }if $\left\| b_{n}\right\| =\left\| b_{n}^{\prime
}\right\| =1$ for every $n\in \omega $;

\item \emph{boundedly complete }if the series $\sum\limits_{n\in \omega
}\lambda _{n}b_{n}$ converges in $Z$ whenever it has uniformly bounded
partial sums.
\end{itemize}

We say that a positive real number $K>0$ is a \emph{basis constant} for $%
(b_{n})_{n\in \omega }$ if we have $\left\|\sum_{i\in n}\langle x,
b_{i}^{\prime }\rangle b_i\right\|\leq K \|x\|$ for all $x\in Z$ and all $%
n\in\omega$. It is furthermore an \emph{unconditional basis constant }if $%
\left\|\sum_{i\in A}\langle x, b_{i}^{\prime }\rangle b_i\right\|\leq K
\|x\| $ for all $x\in Z$ and every finite subset $A$ of $\omega$. Every
basis has a basis constant $K$ \cite[Theorem 3.1]{carothers_short_2005}, and
every unconditional basis has an unconditional basis constant \cite[%
Proposition 4.2.29]{megginson_introduction_1998}. By \cite[Theorem~7.4]%
{carothers_short_2005}, every basis of a reflexive Banach space is boundedly
complete.

All Borel spaces will be \emph{standard}. For a standard Borel space $X$, we
denote by $B(X)$ the space of complex-valued bounded Borel functions on $X$,
and by $\mathcal{B}(X)$ the $\sigma $-algebra of Borel subsets of $X$. For a
Borel measure $\mu $ on a standard Borel space $X$, we denote by $\mathcal{B}%
_{\mu }$ the \emph{measure algebra} of $\mu $. This is the quotient of the
Boolean algebra $\mathcal{B}(X)$ of Borel subsets of $X$ by the ideal of $%
\mu $-null Borel subsets. By \cite[Exercise 17.44]{kechris_classical_1995} $%
\mathcal{B}_{\mu }$ is a complete Boolean algebra. The characteristic
function of a set $F$ will be denoted by $\chi _{F}$.

Given a measure space $(X,\mu )$ and a Hölder exponent $p\in (1,\infty )$,
we will denote the Lebesgue space $L^{p}(X,\mu )$ simply by $L^{p}(\mu )$.
Recall that $L^{p}(\mu )$ is separable precisely if there is a $\sigma $%
-finite Borel measure $\lambda $ on a standard Borel space $Z$ such that $%
L^{p}(\lambda )$ is isometrically isomorphic to $L^{p}(\mu )$. Moreover
there exists $n\in \omega \cup \{\omega \}$ such that $( Z,\lambda) $ is 
Borel-isomorphic to $([0,1]\sqcup n,\nu \sqcup c_{n})$, where $%
\nu $ is the Lebesgue measure on $[0,1]$. The push-forward of a measure $\mu 
$ under a function $\phi $ will be denote by $\phi _{\ast }\mu $ or $\phi
_{\ast }(\mu )$.

If $X$ and $Z$ are Borel spaces, we say that $Z$ is \emph{fibred} over $X$
if there is a Borel surjection $q\colon Z\rightarrow X$. In this case, we
call $q$ the \emph{fiber map}. A \emph{section }of $Z$ is a map $\sigma
\colon X\rightarrow Z$ such that $q\circ \sigma $ is the identity map of $X$%
. For $x\in X$, we denote the value of $\sigma $ at $x$ by $\sigma _{x}$,
and the fiber $q^{-1}(\{x\})$ over $x$ is denoted by $Z_{x}$. If $Z^{(0)}$
and $Z^{(1)}$ are Borel spaces fibred over $X$ via fiber maps $q^{(0)}$ and $%
q^{(1)}$ respectively, then their \emph{fiber product} $Z^{(0)}\ast Z^{(1)}$
is the Borel space fibred over $X$ defined by $Z^{(0)}\ast
Z^{(1)}=\{(z^{(0)},z^{(1)})\colon q^{(0)}(z^{(0)})=q^{(1)}(z^{(1)})\}$.

If $E$ and $F$ are Banach spaces, we will denote by $B(E,F)$ the Banach
space of bounded linear maps from $E$ to $F$. When $E=F$, we abbreviate $%
B(E,E)$ to just $B(E)$. Despite the apparent notational conflict with the
set of Borel functions $B(X)$ on a measurable space $X$, confusion will be
unlikely to arise, and it should always be clear from the context what $%
B(\cdot )$ means. Given a Banach space $E$, its dual space will be denoted
by $E^{\prime }$. Similarly, if $T\colon E\rightarrow F$ is a bounded linear
operator between Banach spaces $E$ and $F$, its dual map will be denoted $%
T^{\prime }\colon F^{\prime }\rightarrow E^{\prime }$. Finally, if $p\in
(1,\infty )$, we will write $p^{\prime }$ for its conjugate Hölder exponent,
which satisfies $\frac{1}{p}+\frac{1}{p^{\prime }}=1$. (We will reserve the
letter $q$ for fiber maps.) We exclude $p=1$ in our analysis mostly for
convenience, because we use duality in many situations. We do not know
whether the results of this paper carry over to the case $p=1$.






\section{Borel Bundles of Banach spaces}

\label{Section: Banach bundles}

\begin{definition}
\label{definition Banach bundles}

Let $X$ be a Borel space. A (standard)\emph{\ Borel Banach bundle} over $X$
is a Borel space $\mathcal{Z}$ fibred over $X$ together with

\begin{enumerate}
\item Borel maps $+\colon \mathcal{Z}\ast \mathcal{Z}\rightarrow \mathcal{Z}$%
, $\cdot \colon \mathbb{C}\times \mathcal{Z}\rightarrow \mathcal{Z}$, and $%
\| \cdot \| \colon \mathcal{Z}\rightarrow \mathbb{C}$,

\item a Borel section $\mathbf{0}\colon X\rightarrow \mathcal{Z}$,

\item a Borel function $d\colon X\rightarrow \omega \cup \{ \omega \} $, and

\item a sequence $\left( \sigma _{n}\right) _{n\in \omega }$ of Borel
sections $\sigma _{n}\colon X\rightarrow \mathcal{Z}$
\end{enumerate}

such that the following holds:

\begin{itemize}
\item $\mathcal{Z}_{x}$ is a reflexive Banach space of dimension $d(x)$
with zero element $\mathbf{0}_{x}$ for every $x\in X$;

\item there is $K>0$ such that, for every $x\in X$, the sequence $\left(
\sigma _{n,x}\right) _{n\in d(x)}$ is a basis of $\mathcal{Z}_{x}$ with
basis constant $K$, and the sequence $(\sigma _{n,x}^{\prime })_{n\in d(x)}$
is a basis of $\mathcal{Z}_{x}^{\prime }$ with basis constant $K$.
\end{itemize}
\end{definition}

In the definition above $(\sigma _{n,x}^{\prime })_{n\in d(x)}$ is the
sequence of coefficient functionals associated with the basis $(\sigma
_{n,x})_{n\in d(x)}$ of $\mathcal{Z}_{x}$. We say that the sequence $%
(\sigma _{n})_{n\in \omega }$ is a \emph{basic sequence} for $\mathcal{Z}$
with \emph{basis constant} $K$. If furthermore there exists $K>0$ such that
for every $x\in X$, the sequences $(\sigma _{n,x})_{n\in d(x)}$ and $\left(
\sigma _{n,x}^{\prime }\right) _{n\in d(x)}$ are unconditional bases of $%
\mathcal{Z}_{x}$ and $\mathcal{Z}_{x}^{\prime }$ with unconditional basis
constant $K$, then we say that $(\sigma _{n})_{n\in \omega }$ is an \emph{%
unconditional basic sequence }with unconditional basis constant $K$.
Finally, we say that $(\sigma _{n})_{n\in \omega }$ is a \emph{normal basic
sequence} if $\left\Vert \sigma _{n,x}\right\Vert =\left\Vert \sigma
_{n,x}^{\prime }\right\Vert =1$ for every $n\in \omega $ and $x\in X$.

\begin{example}[Constant bundles]
Let $X$ be a Borel space, let $Z$ be a reflexive Banach space, and set $%
\mathcal{Z}=X\times Z$. Then $\mathcal{Z}$ with the product Borel structure
is naturally a Borel Banach bundle, where each fiber $\mathcal{Z}_{x}$ is
isomorphic to $Z$. In the particular case when $Z$ is the field of complex
numbers, this is called the \emph{trivial bundle} over $X$.
\end{example}

\begin{example}[Disjoint unions]
\label{Example:disjoint-union}Suppose that, for $i\in \left\{ 0,1\right\} $, 
$\mathcal{Z}_{i}$ is a Borel Banach bundle on the Borel space $X_{i}$. Then
one can consider the disjoint union $X_{0}\sqcup X_{1}$, which is endowed
with a canonical Borel structure. One can then consider the \emph{disjoint
union bundle} $\mathcal{Z}:=\mathcal{Z}_{0}\sqcup \mathcal{Z}_{1}$, which is
a Borel Banach bundle over $X_{0}\sqcup X_{1}$ defined by $\mathcal{Z}_{x}:=(%
\mathcal{Z}_{i})_{x}$ for $x\in X_{i}$.
\end{example}

\begin{remark}
\label{Remark:dimension-fibers}In view of Example \ref%
{Example:disjoint-union}, in the development of the theory of Borel Banach
bundles one can assume without loss of generality that the fibers $\mathcal{Z%
}_{x}$ of the bundle $\mathcal{Z}$ have fixed dimension $d$ independent of $%
x\in X$. Indeed, an arbitrary Borel Banach bundle is a disjoint union in the
sense of Example \ref{Example:disjoint-union} of bundles with fibers with
fixed dimension. In view of this observation, we will consider in the
following Borel Banach bundles with fiber of a fixed dimension $d\in \omega
\cup \left\{ \omega \right\} $. Furthermore, to fix the ideas we only
consider the case when $d$ is infinite, which is the most interesting case.
\end{remark}

Let $q\colon \mathcal{Z}\rightarrow X$ be a Borel Banach bundle. Then the
space of Borel sections of $\mathcal{Z}$ has a natural structure of $B(X)$%
-module. Accordingly, if $\xi _{1}$ and $\xi _{2}$ are Borel sections of $%
\mathcal{Z}$ and $f\in B(X)$, we denote by $\xi _{1}+\xi _{2}$ and $f\xi $
the Borel sections given by $\left( \xi _{1}+\xi _{2}\right) _{x}=(\xi
_{1})_{x}+(\xi _{2})_{x}\ $and $\left( f\xi \right) _{x}=f(x)\xi _{x}$ for
every $x$ in $X$. If $E$ is a Borel subset of $X$, then $q^{-1}(E)$ is
canonically a Borel Banach bundle over $E$, called the \emph{restriction }of 
$\mathcal{Z}$ to $E$, and denoted by $\mathcal{Z}|_{E}$.

\begin{remark}
A Borel Banach bundle where each fiber is a Hilbert space is called a \emph{%
Borel Hilbert bundle}. Such bundles (usually called just Hilbert bundles)
are the key notion in the study of representation of groupoids on Hilbert
spaces; see \cite[Appendix F]{williams_crossed_2007}, \cite[Section 3.1]%
{paterson_groupoids_1999}, and \cite[Section 2]{ramsay_virtual_1971}. The
Gram-Schmidt process shows that a Borel Hilbert bundle $\mathcal{H}$ over $X$
always has a basic sequence $(\sigma _{n})_{n\in \omega }$ such that for all 
$x$ in $X$, the sequence $(\sigma _{n,x})_{n\in \omega }$ is an \emph{%
orthonormal basis }of $\mathcal{H}_{x}$.
\end{remark}

A similar formalism is used by Lafforgue to study semi-continuous bundles of
Banach spaces in duality \cite%
{lafforgue_k-theorie_2002,lafforgue_k-theorie_2007}.

\subsection{Canonical Borel structures\label{Subsection: canonical Borel
structures}}

Let $X$ be a Borel space, and let $\mathcal{Z}$ be a set (with no Borel
structure) fibred over $X$. Assume there are operations%
\begin{equation*}
+\colon \mathcal{Z}\ast \mathcal{Z}\rightarrow \mathcal{Z}\ \ ,\ \ \cdot
\colon \mathbb{C}\times \mathcal{Z}\rightarrow \mathcal{Z}\ \ \mbox{ and }\
\ \Vert \cdot \Vert \colon \mathcal{Z}\rightarrow \mathbb{C},
\end{equation*}%
making each fiber a \emph{reflexive }Banach space. In this situation, we
will say that $\mathcal{Z}$ is a \emph{bundle of Banach spaces} over $X$,
and will denote it by $\bigsqcup\nolimits_{x\in X}\mathcal{Z}_{x}$. Let $%
\mathcal{Z}^{\prime }$ be the set of pairs $\left( x,v\right) $ for $x\in X$
and $v\in \mathcal{Z}_{x}^{\prime }$. Then $\mathcal{Z}^{\prime }$ is also a
bundle of Banach spaces over $X$. Suppose further that there exist $K>0$ and
a sequence $(\sigma _{n})_{n\in \omega }$ of sections $\sigma _{n}\colon
X\rightarrow \mathcal{Z}$ such that, for every $x\in X$, the sequence $%
(\sigma _{n,x})_{n\in \omega }$ is a basis of $\mathcal{Z}_{x}$ with
associated sequence of coefficient functionals $(\sigma _{n,x}^{\prime
})_{n\in \omega }$ and with basis constant $K$. Assume that for every $m\in
\omega $ and every sequence $(\alpha _{j})_{j\in m}$ in $\mathbb{Q}%
(i)^{\oplus m}$, the map $X\rightarrow \mathbb{R}$ given by  $x\mapsto 
\left\Vert \sum\limits_{j\in m}\alpha _{j}\sigma _{j,x}\right\Vert$ 
is Borel. Set 
\begin{equation*}
Z=\left\{ \left( x,(\alpha _{n})_{n\in \omega }\right) \in X\times \mathbb{C}%
^{\omega }\colon \sup_{m\in \omega }{}\left\Vert \sum\limits_{j\in m}\alpha
_{j}\sigma _{j,x}\right\Vert <\infty \right\} \text{.}
\end{equation*}%

We claim that $Z$ is a Borel subset of $X\times \mathbb{C}^{\omega }$. To
see this, note that a pair $\left( x,(\alpha _{n})_{n\in \omega }\right) $
in $X\times \mathbb{C}^{\omega }$ belongs to $Z$ if and only if there is $%
N\in \omega $ such that for every $m,k\in \omega $ there is $(\beta
_{j})_{j\in m}$ in $\mathbb{Q}(i)^{\oplus m}$ such that $\max_{j\in
n}\left\vert \alpha _{j}-\beta _{j}\right\vert \leq \frac{1}{2^{k}}\ $and $%
\left\Vert \sum\nolimits_{j\in m}\beta _{j}\sigma _{j,x}\right\Vert {}<N$.
In other words, $Z$ can be written as%
\begin{equation*}
Z=\bigcup_{N\in \omega }\bigcap_{m,k\in \omega }\bigcup_{\left( \beta
_{j}\right) _{j\in m}\in \mathbb{Q}(i)^{\oplus m}}Z(N,m,k,\left( \beta
_{j}\right) _{j\in m})\text{,}
\end{equation*}%
where $Z(N,m,k,\left( \beta _{j}\right) _{j\in m})$ is the set of pairs $%
\left( x,(\alpha _{n})_{n\in \omega }\right) $ in $X\times \mathbb{C}%
^{\omega }$ such that $\max_{j\in n}\left\vert \alpha _{j}-\beta
_{j}\right\vert \leq \frac{1}{2^{k}}\ $and $X\times \mathbb{C}^{\omega }$
and $\left\Vert \sum\nolimits_{j\in m}\beta _{j}\sigma _{j,x}\right\Vert
{}<N $. Since the map $x\mapsto \left\Vert \sum\nolimits_{j\in m}\beta
_{j}\sigma _{j,x}\right\Vert $ is Borel, $Z(N,m,k,\left( \beta _{j}\right)
_{j\in m})$ is a Borel subset of $X\times \mathbb{C}^{\omega }$. Since Borel
sets form a $\sigma $-algebra, this shows that $Z$ is Borel.

The assignment 
\begin{equation*}
\left( x,(\alpha _{n})_{n\in \omega }\right) \mapsto \sum\limits_{n\in
\omega }\alpha _{n}\sigma _{n,x}
\end{equation*}
induces a bijection $Z\rightarrow \mathcal{Z}$ since, for every $x\in X$,
the sequence $(\sigma _{n,x})_{n\in \omega }$ is a boundedly complete basis
of $\mathcal{Z}_{x}$ \cite[Definition 4.4.8]{megginson_introduction_1998}.
(We are using here the fact that every basis of a reflexive Banach space is
boundedly complete \cite[Theorem 4.4.15]{megginson_introduction_1998}.) This
bijection induces a standard Borel structure on $\mathcal{Z}$, and it is not
difficult to verify that such a Borel structure turns $\mathcal{Z}$ into a
Borel Banach bundle. A similar argument shows that the set%
\begin{equation*}
Z^{\prime }=\left\{ \left( x,(\alpha _{n})_{n\in \omega }\right) \in X\times 
\mathbb{C}^{\omega }\colon \sup_{m\in \omega }{}\left\Vert \sum\limits_{j\in
m}\alpha _{j}\sigma _{j,x}^{\prime }\right\Vert <\infty \right\}
\end{equation*}%
is Borel, and that the map from $Z^{\prime }$ to $\mathcal{Z}^{\prime }$
given by $\left( x,(\alpha _{n})_{n\in \omega }\right) \mapsto
\sum\nolimits_{n\in \omega }\alpha _{n}\sigma _{n,x}^{\prime }$ is a
bijection. This induces a standard Borel structure on $\mathcal{Z}^{\prime }$
that makes $\mathcal{Z}^{\prime }$ a Borel Banach bundle. It follows from
the definition of the Borel structures on $\mathcal{Z}$ and $\mathcal{Z}%
^{\prime }$, that the canonical pairing $\mathcal{Z}\ast \mathcal{Z}^{\prime
}\rightarrow \mathbb{C}$ is Borel. In fact, for $\left( x,(\alpha
_{n})_{n\in \omega }\right) \in Z$ and $\left( x,(\beta _{n})_{n\in \omega
}\right) \in Z^{\prime }$, we have 
\begin{equation*}
\left\langle \sum\limits_{n\in \omega }\alpha _{n}\sigma
_{n,x},\sum\limits_{m\in \omega }\beta _{m}\sigma _{m,x}^{\prime
}\right\rangle =\sum\limits_{n\in \omega }\alpha _{n}\beta _{n}\text{.}
\end{equation*}

The standard Borel structures on $\mathcal{Z}$ and $\mathcal{Z}^{\prime }$
described above will be referred to as the \emph{canonical Borel structures}
associated with the sequence $(\sigma _{n})_{n\in \omega }$ of Borel
sections $X\rightarrow \mathcal{Z}$. By \cite[Theorem~14.12]%
{kechris_classical_1995}, these can be equivalently described as the\emph{\ }%
Borel structures generated by the sequence of functionals on $\mathcal{Z}$
and $\mathcal{Z}^{\prime }$ given by $z\mapsto \left\langle z,\sigma
_{n,q(z)}^{\prime }\right\rangle \ $and$\ \ w\mapsto \left\langle \sigma
_{n,q(w)},w\right\rangle $ for $n\in \omega $.

As a consequence of the previous discussion, we conclude that if $\mathcal{Z}
$ is a Borel Banach bundle, then the Borel structure on $\mathcal{Z}$ is
generated by the sequence of maps $\mathcal{Z}\rightarrow \mathbb{C}$ given
by $z\mapsto \left\langle z,\sigma _{n,q(z)}^{\prime }\right\rangle $ for $n$
in $\omega $. Moreover, the dual bundle $\mathcal{Z}^{\prime }$ has a \emph{%
unique }Borel Banach bundle structure making the canonical pairing Borel. In
the following, whenever $\mathcal{Z}$ is a Borel Banach bundle, we will
always consider $\mathcal{Z}^{\prime }$ as a Borel Banach bundle endowed
with such a canonical Borel structure.\newline
\indent The following criterion to endow a Banach bundle with a Borel
structure is an immediate consequence of the observations contained in this
subsection.

\begin{lemma}
\label{Lemma: canonical Borel structure} Let $(Z_{k})_{k\in \omega }$ be a
sequence of reflexive Banach spaces. For every $k\in \omega $, let $%
(b_{n,k})_{n\in \omega }$ be a basis of $Z_{k}$ with associated sequence of
coefficient functionals $(b_{n,k}^{\prime })_{n\in \omega }$, and suppose
that both $(b_{n,k})_{n\in \omega }$ and $(b_{n,k}^{\prime })_{n\in \omega }$
have basis constant $K$ independent of $k$. Let $\mathcal{Z}$ be a bundle of
Banach spaces over $X$, and assume there exist a Borel partition $%
(X_{k})_{k\in \omega }$ of $X$, and isometric isomorphisms $\psi _{x}\colon
Z_{k}\rightarrow \mathcal{Z}_{x}$ and $\psi _{x}^{\prime }\colon
Z_{k}^{\prime }\rightarrow \mathcal{Z}_{x}^{\prime }$ for $k\in \omega $ and 
$x\in X_{k}$. For $k,n\in \omega $ and $x\in X_{k}$, set $\sigma _{n,x}=\psi
_{x}(b_{n,k})$ and $\sigma _{n,x}^{\prime }=\psi _{x}^{\prime
}(b_{n,k}^{\prime })$. Then there are unique Borel Banach bundle structures
on $\mathcal{Z}$ and $\mathcal{Z}^{\prime }$ such that $(\sigma _{n})_{n\in
\omega }$ and $\left( \sigma _{n}^{\prime }\right) _{n\in \omega }$ are
basic sequences, and such that the canonical pairing between $\mathcal{Z}$
and $\mathcal{Z}^{\prime }$ is Borel.
\end{lemma}

\subsection{Banach space valued \texorpdfstring{$L^p$}{Lp}-spaces\label%
{Section: Banach Lp spaces}}

For the remainder of this section, we fix a Borel Banach bundle $q\colon 
\mathcal{Z}\rightarrow X$ over the standard Borel space $X$, a basic
sequence $(\sigma _{n})_{n\in \omega }$ of $\mathcal{Z}$ with basis constant 
$K$, a $\sigma $-finite Borel measure $\mu $ on $X$, and a Hölder exponent $%
p\in (1,\infty )$. 
Denote by $\mathcal{L}^{p}(X,\mu ,\mathcal{Z})$ the space of Borel sections $%
\xi \colon X\rightarrow \mathcal{Z}$ such that%
\begin{equation*}
N_{p}(\xi )^{p}=\int \left\| \xi _{x}\right\| ^{p}\ d\mu (x)<\infty .
\end{equation*}%
It follows from the Minkowski inequality that $\mathcal{L}^{p}(X,\mu ,%
\mathcal{Z})$ is a seminormed complex vector space. We denote by $%
L^{p}(X,\mu ,\mathcal{Z})$ the normed space obtained as a quotient of the
seminormed space $\left( \mathcal{L}^{p}(X,\mu ,\mathcal{Z}),N_{p}\right) $.
When $\mathcal{Z}$ is the trivial bundle over $X$, then $L^{p}(X,\mu ,%
\mathcal{Z})$ coincides with the Banach space $L^{p}(X,\mu )$. Consistently,
we will abbreviate $\mathcal{L}^{p}(X,\mu ,\mathcal{Z})$ and $L^{p}(X,\mu ,%
\mathcal{Z})$ to $\mathcal{L}^{p}(\mu ,\mathcal{Z})$ and $L^{p}(\mu ,%
\mathcal{Z})$, respectively.

The usual Riesz-Fischer type argument \cite[Theorem 4.8]%
{brezis_functional_2011} shows that $L^{p}\left( \mu ,\mathcal{Z}\right) $
is a Banach space; see \cite[Theorem 1.1.8]{hayes_direct}. A standard
argument (see \cite[Theorem 1.1.8]{hayes_direct}) proves the analog \cite[%
Theorem 4.9]{brezis_functional_2011} in this context: convergence in $%
L^{p}\left( \mu ,\mathcal{Z}\right)$ implies that there exists a subsequence
that converges almost everywhere.

As it is customary, we will identify an element of $\mathcal{L}^{p}(\mu ,%
\mathcal{Z})$ with its image in the quotient $L^{p}(\mu ,\mathcal{Z})$. We
will also write $\Vert \cdot \Vert _{p}$, or just $\Vert \cdot \Vert $ if no
confusion is likely to arise, for the norm on $L^{p}(\mu ,\mathcal{Z})$
induced by $N_{p}$.

\begin{proposition}
\label{Proposition: basis} Let $\xi \in L^{p}(\mu,\mathcal{Z}) $. Then:

\begin{enumerate}
\item The function $\left\langle \xi ,\sigma _{n}^{\prime }\right\rangle
\colon X\rightarrow \mathbb{R}$ defined by $x\mapsto \left\langle \xi _{x},\sigma
_{n,x}^{\prime }\right\rangle $ belongs to $\mathcal{L}^{p}(\mu )$;

\item The sequence $\left(\sum\limits_{k\in n }\left\langle \xi ,\sigma
_{k}^{\prime}\right\rangle \sigma _{k}\right)_{n\in\omega}$ converges to $%
\xi $.
\end{enumerate}
\end{proposition}

\begin{proof}
(1). The function $\left\langle \xi ,\sigma _{n}^{\prime }\right\rangle $ is
Borel because the canonical pairing map is Borel. Moreover, the estimate%
\begin{equation*}
\int \left\vert \left\langle \xi _{x},\sigma _{n,x}^{\prime }\right\rangle
\right\vert ^{p}\ d\mu (x)\leq \left( 2K\right) ^{p}\int \left\Vert \xi
_{x}\right\Vert ^{p}\ d\mu (x)=\left( 2K\Vert \xi \Vert \right) ^{p}
\end{equation*}%
shows that $\left\langle \xi ,\sigma _{n}^{\prime }\right\rangle$ belongs to 
$\mathcal{L}^{p}(\mu)$.

\indent(2). For every $x\in X$, and using that $K$ is a basis constant for $%
(\sigma _{n,x})_{n\in \omega }$, we have%
\begin{equation*}
\left\Vert \sum\limits_{k\in n}\left\langle \xi _{x},\sigma _{k,x}^{\prime
}\right\rangle \sigma _{k,x}\right\Vert \leq K\left\Vert \xi _{x}\right\Vert
.
\end{equation*}%
Given $\varepsilon >0$ and $n\in \omega $, define the Borel set%
\begin{equation*}
F_{n,\varepsilon }=\left\{ x\in X\colon \left\Vert \sum\limits_{k\in
n}\left\langle \xi _{x},\sigma _{k,x}^{\prime }\right\rangle \sigma
_{k,x}-\xi _{x}\right\Vert \leq \varepsilon \right\} \text{.}
\end{equation*}%
Then $\bigcup\limits_{n\in \omega }F_{n,\varepsilon }=X$. By the dominated
convergence theorem, there is $n_{0}\in \omega $ such that%
\begin{equation*}
\int_{X\backslash F_{n_{0},\varepsilon }}\left\Vert \xi _{x}\right\Vert
^{p}\ d\mu (x)<\varepsilon \text{.}
\end{equation*}%
Thus, for $n\geq n_{0}$, we have 
\begin{align*}
\left\Vert \sum\limits_{k\in n}\left\langle \xi ,\sigma _{k}^{\prime
}\right\rangle \sigma _{k}-\xi \right\Vert _{p}^{p}& =\int \left\Vert
\sum\limits_{k\in n}\left\langle \xi _{x},\sigma _{k,x}^{\prime
}\right\rangle \sigma _{k,x}-\xi _{x}\right\Vert ^{p}\ d\mu (x) \\
& \leq \mu \left( F_{n,\varepsilon }\right) \varepsilon +\left( K+1\right)
^{p}\int_{X\left\backslash F_{n,\varepsilon }\right. }\left\Vert \xi
_{x}\right\Vert ^{p}d\mu (x)\\
&\leq \left( \left( K+1\right) ^{p}+1\right)
\varepsilon .
\end{align*}%
This shows that the sequence $\left( \sum\limits_{k\in n}\left\langle \xi
,\sigma _{k}^{\prime }\right\rangle \sigma _{k}\right) _{n\in \omega }$
converges to $\xi $.
\end{proof}

In view of Proposition \ref{Proposition: basis}, the sequence $\left( \sigma
_{n}\right) _{n\in \omega }$ can be thought as a basis of $L^{p}(\mu ,%
\mathcal{Z})$ over $L^{p}(\mu )$. In particular, Proposition \ref%
{Proposition: basis} implies that $L^{p}(\mu ,\mathcal{Z})$ is a \emph{%
separable} Banach space. It is not difficult to verify that, if $(\sigma
_{n})_{n\in \omega }$ is an \emph{unconditional }basic sequence for $%
\mathcal{Z}$, then the series $\sum\nolimits_{k\in \omega }\left\langle \xi
,\sigma _{k}^{\prime }\right\rangle \sigma _{k}$ converges \emph{%
unconditionally }to $\xi $ for every $\xi \in L^{p}(\mu ,\mathcal{Z})$.
Furthermore, since $\left( \sigma _{n,x}\right) _{n\in \omega }$ is a \emph{%
boundedly complete} basis of $\mathcal{Z}_{x}$ for every $x\in X$, a
standard argument shows that $\left( \sigma _{n}\right) _{n\in \omega }$ has
as similar property as basis of $L^{p}\left( \mu ,\mathcal{Z}\right) $ over $%
L^{p}\left( \mu \right) $. In other words, if the series $\sum_{k\in \omega
}f_{k}\sigma _{k}$ has uniformly bounded partial sums, then it converges in $%
L^{p}(\mu ,\mathcal{Z})$.

\subsection{Pairing\label{Subsection: pairing}}

In this subsection, we show that there is a natural pairing between $%
L^{p}(\mu ,\mathcal{Z})$ and $L^{p^{\prime }}(\mu ,\mathcal{Z}^{\prime })$,
under which we may identify $L^{p}(\mu ,\mathcal{Z})^{\prime }$ with $%
L^{p^{\prime }}(\mu ,\mathcal{Z}^{\prime })$. Define a map 
\begin{equation*}
\langle \cdot ,\cdot \rangle \colon L^{p}(\mu ,\mathcal{Z})\times
L^{p^{\prime }}(\mu ,\mathcal{Z}^{\prime })\rightarrow \mathbb{C}\ \ 
\mbox{
by }\ \ \langle \xi ,\eta \rangle =\int \langle \xi _{x},\eta _{x}\rangle \
d\mu (x)
\end{equation*}%
for all $\xi \in L^{p}(\mu ,\mathcal{Z})$ and all $\eta \in L^{p^{\prime
}}(\mu ,\mathcal{Z}^{\prime })$. Young's inequality shows that the
assignment $x\mapsto \langle \xi _{x},\eta _{x}\rangle $ is integrable, and
hence the map above is well defined. The following is the analog in this
context of the classical Riesz representation theorem \cite[Theorem 4.11]%
{brezis_functional_2011}, and can be proved with similar methods.

\begin{theorem}
\label{Theorem: dual} The function from $L^{p^{\prime }}\left( \mu ,\mathcal{%
Z}^{\prime }\right) $ to $L^{p}\left( \mu ,\mathcal{Z}\right) ^{\prime }$
given by 
\begin{equation*}
\eta \mapsto \langle \cdot ,\eta \rangle =\int \langle \cdot _{x},\eta
_{x}\rangle \ d\mu (x)
\end{equation*}%
is an isometric isomorphism.
\end{theorem}

It follows in particular that the Banach space $L^{p}(\mu ,\mathcal{Z})$ is
reflexive. (Recall that all the fibres of the Banach bundle $\mathcal{Z}$
are assumed to be reflexive Banach spaces.)

\subsection{Bundles of \texorpdfstring{$L^p$}{Lp}-spaces\label{Subsection:
Lp-bundles}}

Consider a Borel probability measure $\mu $ on a standard Borel space $X$.
Let $\lambda $ be a Borel probability measure on a standard Borel space $Z$
fibred over $X$ via a fiber map $q$ such that $q_{\ast }(\lambda )=\mu $. By 
\cite[Exercise 17.35]{kechris_classical_1995}, the measure $\lambda $ admits
a disintegration $(\lambda _{x})_{x\in X}$ with respect to $\mu $, which is
also written as $\lambda =\int \lambda _{x}\ d\mu (x)$. In other words,

\begin{itemize}
\item there is a Borel assignment $x\mapsto \lambda _{x}$, where $\lambda
_{x}$ is a probability measure on $\mathcal{Z}_{x}$, and

\item for every bounded Borel function $f\colon Z\rightarrow \mathbb{C}$, we
have $\int f\ d\lambda =\int \ d\mu (x)\int f\ d\lambda _{x}$.
\end{itemize}

Consider the Banach bundle $\mathcal{Z}=\bigsqcup\nolimits_{x\in
X}L^{p}(\lambda _{x})$ over $X$, where the fiber $\mathcal{Z}_{x}$ over $x$
is $L^{p}(\lambda _{x})$.

\begin{theorem}
\label{thm: BBb structure} There is a canonical Borel Banach bundle
structure on $\mathcal{Z}$ such that $L^p(\mu,\mathcal{Z})$ is isometrically
isomorphic to $L^p(\lambda)$.
\end{theorem}

\begin{proof}
Let us assume for simplicity that $\mu $ and $\lambda _{x}$ are atomless for
every $x\in X$. In this case, by \cite[Theorem~2.2]{graf_classification_1989}%
, we can assume without loss of generality that

\begin{itemize}
\item $X$ is the unit interval $[ 0,1] $ and $\mu$ is its Lebesgue measure;

\item $Z$ is the unit square $[ 0,1] ^{2}$ and $\lambda$ is its Lebesgue
measure;

\item $q\colon Z\to X$ is the projection onto the first coordinate; and

\item $\lambda _{x}$ is the Lebesgue measure on $\{ x\} \times [ 0,1] $ for
every $x\in X$.
\end{itemize}

Let $(h_{n})_{n\in \omega }$ be the Haar system on $[0,1]$ defined as in 
\cite[Chapter 3]{carothers_short_2005}. For $n\in \omega $ and $x\in \lbrack
0,1]$, define $h_{n,x}^{(p)}\colon \lbrack 0,1]\rightarrow \mathbb{R}$ by $%
h_{n,x}^{(p)}(t)=h_{n}(t)/\left\| h_{n}\right\| _{p}$ for every $t\in
\lbrack 0,1]$. Then $(h_{n,x}^{(p)})_{n\in \omega }$ is a normalized basis
of $L^{p}(\lambda _{x})$ for every $x\in \lbrack 0,1]$. It follows from the
discussion in Subsection \ref{Subsection: canonical Borel structures} that
there are unique Borel Banach bundle structures on $\mathcal{Z}$ and $%
\mathcal{Z}^{\prime }=\bigsqcup\nolimits_{x\in X}L^{p^{\prime }}(\lambda
_{x})$ such that $(h_{n}^{(p)})_{n\in \omega }$ and $(h_{n}^{(p^{\prime
})})_{n\in \omega }$ are normal basic sequences for $\mathcal{Z}$ and $%
\mathcal{Z}^{\prime }$, and that the canonical pairing between $\mathcal{Z}$
and $\mathcal{Z}^{\prime }$ is Borel.

We claim that $L^{p}(\mu ,\mathcal{Z})$ can be canonically identified with $%
L^{p}(\lambda )$. Given $f\in L^{p}(\lambda )$, consider the Borel section $%
s_{f}\colon X\rightarrow \mathcal{Z}$ defined by $s_{f,x}(t)=f(x,t)$ for $%
x,t\in \lbrack 0,1]$. It is clear that $s_{f,x}$ belongs to $L^{p}(\mu ,%
\mathcal{Z})$ and that $\int \left\| s_{f,x}\right\| _{p}^{p}\ d\mu
(x)=\left\| f\right\| _{p}^{p}$. It follows that the map $f\mapsto s_{f,x}$
induces an isometric linear map $s\colon L^{p}(\lambda )\rightarrow
L^{p}(\mu ,\mathcal{Z})$. The fact that $s$ is surjective is a consequence
of Proposition \ref{Proposition: basis}, since the range of $s$ is a closed
linear subspace of $L^{p}(\mu ,\mathcal{Z})$ that contains $h_{n}^{(p)}$ for
every $n\in \omega $.

The case when $\lambda $ and $\mu $ are arbitrary Borel probability measures
can be treated similarly, using the classification of disintegration of
Borel probability measures given in \cite[Theorem~3.2]%
{graf_classification_1989}, together with Lemma~\ref{Lemma: canonical Borel
structure}. In fact, the results of \cite{graf_classification_1989} show
that the same conclusions hold if $\lambda $ is a Borel $\sigma $\emph{%
-finite }measure.
\end{proof}

\begin{definition}
\label{Definition: Lp-bundle} Let $X$ be a Borel space, and let $\mu $ be a
Borel probability measure on $X$. An $L^{p}$\emph{-bundle} over $(X,\mu )$
is a Borel Banach bundle $\mathcal{Z}=\bigsqcup\nolimits_{x\in
X}L^{p}(\lambda _{x})$ obtained from the disintegration of a $\sigma $%
-finite Borel measure $\lambda $ on a Borel space $Z$ fibred over $X$, as
described in Theorem~\ref{thm: BBb structure}.
\end{definition}

\subsection{Decomposable operators\label{Subsection: decomposable operators}}

Throughout this section we let $q_{X}\colon \mathcal{Z}\rightarrow X$ and $%
q_{Y}\colon \mathcal{W}\rightarrow Y$ be standard Borel Banach bundles with
basic sequences $(\sigma _{n})_{n\in \omega }$ and $(\tau _{n})_{n\in \omega
}$, respectively, and we let $\phi \colon X\rightarrow Y$ be a Borel
isomorphism.

\begin{definition}
Let $B( \mathcal{Z},\mathcal{W},\phi) $ be the space of bounded linear maps
of the form $T\colon \mathcal{Z}_{x}\to \mathcal{W}_{\phi (x) }$ for some $%
x\in X$. For such a map $T$, we denote the corresponding point $x$ in $X$ by 
$x_T$.
\end{definition}

Consider the Borel structure on $B( \mathcal{Z},\mathcal{W},\phi) $
generated by the maps $T\mapsto x_T$ and $T \mapsto \left\langle T\sigma
_{n,x_T },\tau _{m,\phi(x_T) }^{\prime}\right\rangle$ for $n,m\in \omega $.
It is not difficult to check that the operator norm and composition of
operators are Borel functions on $B( \mathcal{Z},\mathcal{W},\phi)$, which
make $B( \mathcal{Z},\mathcal{W},\phi) $ into a Borel space fibred over $X$.

\begin{lemma}
The Borel space $B( \mathcal{Z},\mathcal{W},\phi) $ is standard.
\end{lemma}

\begin{proof}
Let $V$ be the set of elements $\left( x,(c_{n,k})_{n,m\in \omega }\right)$
in $X\times \mathbb{C}^{\omega \times \omega }$ such that, for some $M\in
\omega $ and every $(\alpha _{n})_{n\in \omega }\in \mathbb{Q}(i)^{\oplus
\omega }$, we have 
\begin{equation*}
\sup_{m\in \omega }\left\| \sum\limits_{k\in m}\left( \sum\limits_{n\in
\omega }\alpha_{n}c_{n,m}\right) \tau _{\phi (x),m}\right\| \leq M\sup_{n\in
\omega }\left\| \sum\limits_{k\in n}\alpha _{k}\sigma _{x,k}\right\| \text{.}
\end{equation*}%
Then $V$ is a Borel subset of $X\times \mathbb{C}^{\omega \times \omega }$,
and it is therefore a standard Borel space by \cite[Corollary~13.4]%
{kechris_classical_1995}. The result follows since the function $B(\mathcal{Z%
},\mathcal{W},\phi )\rightarrow X\times \mathbb{C}^{\omega \times \omega }$
given by 
\begin{equation*}
T\mapsto \left( x_{T},\left( \left\langle T\sigma _{n,x_{T}},\tau _{m,\phi
(x_{T})}^{^{\prime }}\right\rangle \right) _{\left( n,m\right) \in \omega
\times \omega }\right)
\end{equation*}%
is a Borel isomorphism between $B(\mathcal{Z},\mathcal{W},\phi )$ and $V$.
\end{proof}

Fix Borel $\sigma $-finite measures $\mu $ on $X$ and $\nu $ on $Y$ with $%
\phi _{\ast }(\mu )\sim \nu $. Suppose that $x\mapsto T_{x}$ is a Borel
section of $B(\mathcal{Z},\mathcal{W},\phi )$ such that, for some $M\geq 0$
and $\mu $-almost every $x\in X$, we have 
\begin{equation}
\left\| T_{x}\right\| ^{p}\leq M^{p}\frac{d\phi _{\ast }(\mu )}{\ d\nu }%
(\phi (x))\text{\label{Equation: norm decomposable}.}
\end{equation}%
Then one can define a bounded linear operator $T\colon L^{p}(\mu ,\mathcal{Z}%
)\rightarrow L^{p}(\nu ,\mathcal{W})$ by setting $(T\xi )_{y}=T_{\phi
^{-1}(y)}\xi _{\phi ^{-1}(y)}$ for all $y\in Y$. It is not hard to verify
that $T$ is indeed bounded, with norm given by the infimum of the $M>0$ for
which \eqref{Equation: norm decomposable} holds for $\mu $-almost every $%
x\in X$. Operators of this form are called \emph{decomposable} with respect
to the Borel isomorphism $\phi \colon X\rightarrow Y$. The Borel section $%
x\mapsto T_{x}$ corresponding to the decomposable operator $T$ is called the 
\emph{disintegration }of $T$ with respect to the Borel isomorphism $\phi
\colon X\rightarrow Y$.

\begin{remark}
\label{Remark:uniqueness}It is not difficult to verify that the
disintegration of a decomposable operator $T$ is essentially unique, in the
sense that if $x\mapsto T_{x}$ and $x\mapsto \widetilde{T}_{x}$ are two
Borel sections defining the same decomposable operator, then $T_{x}=%
\widetilde{T}_{x}$ for $\mu $-almost every $x$ in $X$; see for example \cite[%
Lemma F.20]{williams_crossed_2007}.
\end{remark}

Given a bounded Borel function $g\colon Y\rightarrow \mathbb{C}$, we denote
by $\Delta _{g}$ the corresponding multiplication operator on $L^{p}(\nu ,%
\mathcal{W})$. The following characterization of decomposable operators is
the natural generalization of the similar characterization of decomposable
operators on Hilbert bundles \cite[Theorem F.21]{williams_crossed_2007}%
---see also \cite[Theorem 7.10]{takesaki_theory_2002}---and can be proved
with similar methods.

\begin{proposition}
\label{Proposition: characterization decomposable} For a bounded map $%
T\colon L^{p}(\mu,\mathcal{Z}) \to L^{p}( \nu ,\mathcal{W}) $, the following
are equivalent:

\begin{enumerate}
\item $T$ is decomposable with respect to $\phi $;

\item $\Delta _{g}T=T\Delta _{g\circ \phi }$ for every bounded Borel
function $g\colon Y\rightarrow \mathbb{C}$;

\item There is a countable collection $\mathcal{F}$ of Borel subsets of $Y$
that separates the points of $Y$, such that $\Delta _{\chi _{F}}T=T\Delta
_{\chi _{\phi ^{-1}[F]}}$ for every $F\in \mathcal{F}$.
\end{enumerate}
\end{proposition}

\begin{definition}
\label{Definition: phi morphism}A $\phi $-\emph{isomorphism }from $\mathcal{Z%
}$ to $\mathcal{W}$ is a Borel section $x\mapsto T_{x}$ of the bundle $B(%
\mathcal{Z},\mathcal{W},\phi )$ such that $T_{x}$ is a surjective isometry
for every $x\in X$.
\end{definition}

If $T=\left( T_{x}\right) _{x\in X}$ is a $\phi $-isomorphism from $\mathcal{%
Z}$ to $\mathcal{W}$, we denote by $T^{-1}$ the $\phi ^{-1}$-isomorphism $%
(T_{\phi ^{-1}(y)}^{-1})_{y\in Y}$ from $\mathcal{W}$ to $\mathcal{Z}$.

\begin{definition}
\label{Definition: isomorphism bundles} If $X=Y$, then $\mathcal{Z}$ and $%
\mathcal{W}$ are said to be \emph{isomorphic }if there is an $\mbox{id}_{X}$%
-isomorphism from $\mathcal{Z}$ to $\mathcal{W}$. In this case an $\mbox{id}%
_{X}$-isomorphism is simply called an \emph{isomorphism}.
\end{definition}

The classical Guichardet decomposition theorem \cite{paterson_groupoids_1999}%
---see also \cite[page 67]{renault_groupoid_1980}---admits a straightforward
generalization from Hilbert bundles to Banach bundles, which can be proved
by the same method.

\begin{theorem}
\label{Theorem: Guichardet} If $T\colon L^{p}(\mu ,\mathcal{Z})\rightarrow
L^{p}(\nu ,\mathcal{W})$ is an invertible isometry decomposable with respect
to $\phi $, then $T$ admits a disintegration%
\begin{equation*}
x\mapsto \left( \frac{d\phi _{\ast }(\mu )}{d\nu }\phi (x)\right) ^{\frac{1}{%
p}}T_{x}
\end{equation*}%
where $x\mapsto T_{x}$ is a $\phi $-isomorphism from the restriction of $%
\mathcal{Z}$ to a $\mu $-conull Borel set to the restriction of $\mathcal{W}$
to a $\nu $-conull Borel set.
\end{theorem}

\section{Banach representations of étale groupoids}

\subsection{Some background notions on groupoids\label{Subsection:
background on groupoids}}

A \emph{groupoid} can be defined as a (nonempty) small category where every
arrow is invertible. The set of objects of a groupoid $G$ is denoted by $%
G^{0}$. Identifying an object with its identity arrow, one can regard $G^{0}$
as a subset of $G$. We will denote the source and range maps on $G$ by $%
s,r\colon G\rightarrow G^{0}$, respectively. A pair of arrows $\left( \gamma
,\rho \right) $ is \emph{composable} if $s(\gamma )=r(\rho )$. The set of
pairs of composable arrows will be denoted, as customary, by $G^{2}$. If $%
(\gamma ,\rho )$ is a pair of composable arrows of $G$, we denote their
composition by $\gamma \rho $. If $A$ and $B$ are subsets of $G$, we denote
by $AB$ the set of $\gamma \rho $ for $\gamma \in A$ and $\rho \in B$
composable. Similarly, if $A$ is a subset of $G$ and $\gamma \in G$, then we
write $A\gamma $ for $A\{\gamma \}$ and $\gamma A$ for $\{\gamma \}A$. In
particular, when $x$ is an object of $G$, then $Ax$ denotes the set of
elements of $A$ with source $x$, while $xA$ denotes the set of elements of $%
A $ with range $x$. A \emph{bisection }of a groupoid $G$ is a subset $A$ of $%
G$ such that source and range maps are injective on $A$. (Bisections are
called $G$-sets in \cite{renault_groupoid_1980,paterson_groupoids_1999}.) If 
$U\subseteq G^{0}$, then the set of elements of $G$ with source and range in 
$U$ is again a groupoid, called the \emph{restriction }of $G$ to $U$ (or the 
\emph{contraction }in \cite{mackey_ergodic_1963,ramsay_virtual_1971}), and
will be denoted by $G|_{U}$.

A \emph{locally compact groupoid }is a groupoid endowed with a topology
having a countable basis of Hausdorff open sets with compact closures, such
that

\begin{enumerate}
\item composition and inversion of arrows are continuous maps, and

\item the set of objects $G^{0}$, as well as $Gx$ and $xG$ for every $x\in
G^{0}$, are locally compact Hausdorff spaces.
\end{enumerate}

It follows that also source and range maps are continuous, since $s(\gamma
)=\gamma ^{-1}\gamma $ and $r(\gamma )=\gamma \gamma ^{-1}$ for all $\gamma
\in G$. It should be noted that the topology of a locally compact groupoid
might not be (globally) Hausdorff. Examples of non-Hausdorff locally compact
groupoids often arise in the applications, such as the holonomy groupoid of
a foliation; see \cite[Section 2.3]{paterson_groupoids_1999}. Locally
compact groups are the locally compact groupoids with only one object.

\begin{definition}
An \emph{étale groupoid }is a locally compact groupoid such that composition
of arrows---or, equivalently, the source and range maps---are local
homeomorphisms.
\end{definition}

This in particular implies that $Gx$ and $xG$ are countable discrete sets
for every $x\in G^{0}$. Étale groupoids can be regarded as the analog of
countable discrete groups. In fact, countable discrete groups are precisely
the étale groupoids with only one object.

In the following we suppose that $G$ is an étale groupoid. If $U$ is an open
Hausdorff subset of $G$, then $C_{c}(U)$ is the space of compactly supported
continuous functions on $U$. Recall that $B(G)$ denotes the space of
complex-valued Borel functions on $G$. We define $C_{c}(G)$ to be the linear
span inside $B(G)$ of the union of the sets $C_{c}(U)$ for $U$ ranging over
the open Hausdorff subsets of $G$. (Equivalently, $U$ ranges over a covering
of $G$ consisting of open bisections \cite[Proposition 3.10]%
{exel_inverse_2008}.)

\begin{remark}
When $G$ is a \emph{Hausdorff} étale groupoid, then $C_{c}(G)$ as defined
above coincides with the space of compactly supported continuous functions
on $G$.
\end{remark}

One can define the convolution product and involution on $C_{c}(G)$ by%
\begin{equation*}
(f\ast g)(\gamma )=\sum_{\rho _{0}\rho _{1}=\gamma }f(\rho _{0})g(\rho
_{1})\ \ \mbox{ and }\ \ f^{\ast }(\gamma )=\overline{f(\gamma ^{-1})}
\end{equation*}%
for $f,g\in C_{c}(G)$. For $f\in C_{c}(G)$, its $I$-norm is given by%
\begin{equation*}
\left\| f\right\| _{I}=\max \left\{ \sup_{x\in G}\sum_{\gamma \in
xG}\left\vert f(\gamma )\right\vert ,\sup_{x\in G}\sum_{\gamma \in
Gx}\left\vert f(\gamma )\right\vert \right\} .
\end{equation*}%
These operations turn $C_{c}(G)$ into a normed *-algebra; see \cite[Section
2.2]{paterson_groupoids_1999}.

Similarly, one can define the space $B_{c}(G)$ as the linear span inside $%
B(G)$ of the space of complex-valued bounded Borel functions on $G$
vanishing outside a compact Hausdorff subset of $G$. Convolution product,
inversion, and the $I$-norm can be defined exactly in the same way on $%
B_{c}(G)$ as on $C_{c}(G)$, making $B_{c}(G)$ a normed *-algebra; see \cite[%
Section 2.2]{paterson_groupoids_1999}. Both $C_{c}(G)$ and $B_{c}(G)$ have a
contractive approximate identity.

\begin{definition}
\label{Definition:RepresentationCc(G)}A \emph{representation} of $C_{c}(G)$
on a Banach space $Z$ is an algebra homomorphism $\pi \colon
C_{c}(G)\rightarrow B(Z)$. We say that $\pi $ is \emph{contractive} if it is
contractive with respect to the $I$-norm on $C_{c}(G)$. Two representation $%
\pi_0 ,\pi_1$ of $C_{c}( G) $ on $B( Z_1) $ and $B(Z_2)$ are \emph{%
equivalent }if there exists a surjective linear isometry $u\colon
Z_1\rightarrow Z_2$ such that $\pi_2( f) u=u\pi_1 ( f) $ for every $f\in
C_{c}( G) $.
\end{definition}

A Borel probability measure $\mu $ on $G^{0}$ induces $\sigma $-finite Borel
measures $\nu $ and $\nu ^{-1}$ on $G$ given by $\nu
(A)=\int_{G^{0}}\left\vert xA\right\vert \ d\mu (x)$ and $\nu ^{-1}(A)=\nu
(A^{-1})$ for every Borel subset $A$ of $G$. Observe that $\nu $ is the
measure obtained integrating the Borel family $(c_{xG})_{x\in X}$---where $%
c_{xG}$ denotes the counting measure on $xG$---with respect to $\mu $.
Similarly, $\nu^{-1}$ is the measure obtained integrating $(c_{Gx})_{x\in X}$
with respect to $\mu $. The measure $\mu $ is said to be \emph{%
quasi-invariant }if $\nu $ and $\nu ^{-1}$ are equivalent, in symbols $\nu
\sim \nu ^{-1}$. In such case, the Radon-Nikodym derivative $\frac{\ d\nu }{%
\ d\nu ^{-1}}$ will be denoted by $D$. Results of Hahn \cite{hahn_haar_1978}
and Ramsay \cite[Theorem~3.20]{ramsay_topologies_1982} show that one can
always choose---as we will do in the following---$D$ to be a Borel
homomorphism from $G$ to the multiplicative group of strictly positive real
numbers.

\begin{remark}
Étale groupoids can be characterized as those locally compact groupoids
whose topology admits a countable basis of \emph{open bisections}.
\end{remark}

Closely related to the notion of an étale groupoid is that of an inverse
semigroup.

\begin{definition}
An \emph{inverse semigroup} is a semigroup $S$ such that for every element $%
s $ of $S$, there exists a unique element $s^{\ast }$ of $S$ such that $%
ss^{\ast }s=s$ and $s^{\ast }ss^{\ast }=s^{\ast }$.
\end{definition}

Let $G$ be an étale groupoid, and denote by $\Sigma (G)$ the set of open
bisections of $G$. The operations%
\begin{equation*}
AB=\{\gamma \rho \colon (\gamma ,\rho )\in (A\times B)\cap G^{2}\}\ \ 
\mbox{
and }\ \ A^{-1}=\{\gamma ^{-1}\colon \gamma \in A\}
\end{equation*}%
turn $\Sigma (G)$ into an inverse semigroup. The set $\Sigma _{c}(G)$ of 
\emph{precompact }open bisections of $G$ is a subsemigroup of $\Sigma (G)$.
Similarly, the set $\Sigma _{\mathcal{K}}(G)$ of \emph{compact} open
bisections of $G$ is also a subsemigroup of $\Sigma (G)$.

\begin{definition}
\label{Definition:ample}An étale groupoid $G$ is called \emph{ample }if $%
\Sigma _{\mathcal{K}}(G)$ is a basis for the topology of $G$. This is
equivalent to the assertion that $G^{0}$ has a countable basis of compact
open sets.
\end{definition}

\subsection{Representations of étale groupoids on Banach bundles\label%
{Subsections: representations on Banach bundles}}

Throughout the rest of this section, we fix an étale groupoid $G$, and a
Borel Banach bundle $q\colon \mathcal{Z}\rightarrow G^{0}$. We define the 
\emph{groupoid of fiber-isometries} of $\mathcal{Z}$ by 
\begin{equation*}
\mathrm{Iso}(\mathcal{Z})=\left\{ (T,x,y)\colon T\colon \mathcal{Z}%
_{x}\rightarrow \mathcal{Z}_{y}\mbox{ is an invertible isometry, and }x,y\in
G^{0}\right\} .
\end{equation*}%
We denote the elements of $\mathrm{Iso}(\mathcal{Z})$ simply by $T\colon 
\mathcal{Z}_{x}\rightarrow \mathcal{Z}_{y}$. The set $\mathrm{Iso}(\mathcal{Z%
})$ has naturally the structure of groupoid with set of objects $G^{0}$,
where the source and range of the fiber-isometry $T\colon \mathcal{Z}%
_{x}\rightarrow \mathcal{Z}_{y}$ are $s(T)=x$ and $r(T)=y$, respectively. If 
$\left( \sigma _{n}\right) _{n\in \omega }$ is a basic sequence for $%
\mathcal{Z}$, then the Borel structure generated by the maps%
\begin{equation*}
T\mapsto \left\langle T\sigma _{n,s\left( T\right)
},\sigma^{\prime}_{m,r\left( T\right) }\right\rangle
\end{equation*}%
for $n,m\in \omega $, is standard, and makes composition and inversion of
arrows Borel. In other words $\mathrm{Iso}(\mathcal{Z})$ is a \emph{standard
Borel groupoid }\cite[Definition 2.4.1]{lupini_polish_2014}.

Let $\mu $ be a quasi-invariant Borel probability measure on $G^{0}$. A map $%
T\colon G\rightarrow \mathrm{Iso}(\mathcal{Z})$ is said to be a \emph{$\mu $%
-almost everywhere homomorphism}, if there exists a $\mu $-conull Borel
subset $U$ of $G^{0}$ such that the restriction of $T$ to $G|_{U}$ is a
Borel groupoid homomorphism which is the identity on $U$.

\begin{definition}
\label{Definition: representation on Banach bundle} A \emph{representation}
of $G$ on $\mathcal{Z}$ is a pair $(\mu ,T)$ consisting of a quasi-invariant
Borel probability measure $\mu $ on $G^{0}$, and a $\mu $-almost everywhere
homomorphism $T\colon G\rightarrow \mathrm{Iso}(\mathcal{Z})$.
\end{definition}

If $G$ is a discrete group, then a Borel Banach bundle over $G^{0}$ is just
a Banach space $Z$, and a representation of $G$ on $Z$ is a Borel group
homomorphism from $G$ to the Polish group $\mathrm{Iso}(Z) $ of invertible
isometries of $Z$ endowed with the strong operator topology. (It should be
noted here that a Borel group homomorphism from $G$ to $\mathrm{Iso}(Z) $ is
automatically continuous by \cite[Theorem~9.10]{kechris_classical_1995}.)

\begin{definition}
\label{Definition: dual representation} Let $(\mu ,T)$ be a representation
of $G$ on $\mathcal{Z}$. The \emph{dual representation} of $(\mu ,T)$ is the
representation $\left( \mu ,T^{\prime }\right) $ of $G$ on $\mathcal{Z}%
^{\prime }$ defined by%
\begin{equation*}
(T^{\prime })_{\gamma }=(T_{\gamma ^{-1}})^{\prime }\colon \mathcal{Z}%
_{s(\gamma )}^{\prime }\rightarrow \mathcal{Z}_{r(\gamma )}^{\prime }
\end{equation*}%
for all $\gamma \in G$.
\end{definition}

There is a natural notion of equivalence for representations of $G$ on
Banach bundles. Recall that an isomorphism $v:\mathcal{Z}\rightarrow 
\widetilde{\mathcal{Z}}$ between Borel Banach bundles on $X$ is a Borel
section $\left( v_{x}\right) _{x\in X}$ of the Borel Banach bundle $B(%
\mathcal{Z},\widetilde{\mathcal{Z}},\mbox{id}_{X})$ such that $\nu _{x}:%
\mathcal{Z}_{x}\rightarrow \widetilde{\mathcal{Z}}_{x}$ is a surjective
linear isometry; see Definition \ref{Definition: isomorphism bundles}.

\begin{definition}
\label{Definition: equivalence representations} Two representations $(\mu
,T) $ and $(\widetilde{\mu },\widetilde{T})$ of $G$ on Borel Banach bundles $%
\mathcal{Z}$ and $\widetilde{\mathcal{Z}}$ over $G^{0}$, are said to be 
\emph{equivalent}, if $\mu \sim \widetilde{\mu }$ and there are a $\mu $%
-conull Borel subset $U$ of $G^{0}$, and an isomorphism $v\colon \mathcal{Z}%
|_{U}\rightarrow \widetilde{\mathcal{Z}}|_{U}$ such that $\widetilde{T}%
_{\gamma }v_{s(\gamma )}=v_{r(\gamma )}T_{\gamma }$ for every $\gamma \in
G|_{U}$.
\end{definition}

It is clear that two representations are equivalent if and only if their
dual representations as in Definition \ref{Definition: dual representation}
are equivalent.

From now on, we fix a Hölder exponent $p\in (1,\infty )$. Suppose that $(\mu
,T)$ is a representation of $G$ on $\mathcal{Z}$. Then the equation 
\begin{equation}
(\pi _{T}(f)\xi )_{x}=\sum\limits_{\gamma \in xG}f(\gamma )D(\gamma )^{-%
\frac{1}{p}}T_{\gamma }\xi _{s(\gamma )}\text{\label{Equation: integrated
form representation}}
\end{equation}%
for $f\in C_{c}(G)$, $\xi \in L^{p}(\mu ,\mathcal{Z})$, and $x\in G^{0}$,
defines an $I$-norm contractive, nondegenerate representation $\pi
_{T}\colon C_{c}(G)\rightarrow B(L^{p}(\mu ,\mathcal{Z}))$. This can be
proved by proceeding as in the proof of \cite[Proposition 3.1.1]%
{paterson_groupoids_1999}, where $2$ is replaced by $p$, using the duality
result from Theorem~\ref{Theorem: dual}.

\begin{definition}
\label{Definition:integrated-form}Let $(\mu ,T)$ be a representation of $G$
on $\mathcal{Z}$. We call the representation $\pi _{T}\colon
C_{c}(G)\rightarrow B(L^{p}(\mu ,\mathcal{Z}))$ described above the \emph{%
integrated form }of $(\mu ,T)$.
\end{definition}

\begin{remark}
\label{remark: extend to BcG} Given a representation $(\mu ,T)$ of $G$ on $%
\mathcal{Z}$, one can show that there is an $I$-norm contractive
nondegenerate representation $\pi _{T}\colon B_{c}(G)\rightarrow B(L^{p}(\mu
,\mathcal{Z}))$ defined by the same expression as in Equation 
\eqref{Equation:
integrated form representation}.
\end{remark}

Given $f\in C_c(G)$, we let $\check{f}\in C_c(G)$ be given by $\check{f}%
(\gamma)=f(\gamma^{-1})$ for all $\gamma\in G$.

\begin{definition}
\label{Definition:dual}Let $\mu $ be a Borel $\sigma $-finite measure on $%
G^{0}$ and let $\pi \colon C_{c}(G)\rightarrow B(L^{p}(\mu ,\mathcal{Z}))$
be an $I$-norm contractive nondegenerate representation. The \emph{dual
representation} of $\pi $ is the $I$-norm contractive nondegenerate
representation $\pi ^{\prime }\colon C_{c}(G)\rightarrow B(L^{p^{\prime
}}(\mu ,\mathcal{Z}^{\prime }))$ given by $\pi ^{\prime }(f)=\pi (\check{f}%
)^{\prime }$ for all $f\in C_{c}(G)$.
\end{definition}

A straightforward computation shows that the dual representation of $\pi
_{T} $ as in Definition \ref{Definition:dual} is the integrated form of the
dual representation of $T$ from Definition \ref{Definition: dual
representation}. A similar argument as \cite{paterson_groupoids_1999}%
---using the version of Guichardet's decomposition theorem provided by
Theorem~\ref{Theorem: Guichardet}---shows that he assignment $T\mapsto \pi
_{T}$ preserves the natural notions of equivalence of representations
introduced in Definition \ref{Definition:RepresentationCc(G)} and Definition %
\ref{Definition: equivalence representations}.

\begin{proposition}
Let $(\mu ,\mathcal{Z})$ and $(\lambda ,\mathcal{W})$ be Borel Banach
bundles over $G^{0}$, and let $T$ and $S$ be groupoid representations of $G$
on $\mathcal{Z}$ and $\mathcal{W}$, respectively. Then $T$ and $S$ are
equivalent if and only if $\pi _{T}$ and $\pi _{S}$ are equivalent.
\end{proposition}

\subsection{Amplification of representations\label{Subsection: Amplification
of representations}}

Given a natural number $n\geq 1$, regard $M_{n}(C_{c}(G))$ as a normed
*-algebra with respect to the usual matrix product and involution, and the $%
I $-norm%
\begin{equation*}
\left\| [f_{ij}]_{i,j\in n}\right\| _{I}=\max \left\{ \max_{x\in
G^{0}}\max_{j\in n}\sum\limits_{j\in n}\sum\limits_{\gamma \in xG}\left\vert
f_{ij}(\gamma )\right\vert \ ,\ \max_{x\in G^{0}}\max_{i\in
n}\sum\limits_{j\in n}\sum\limits_{\gamma \in Gx}\left\vert f_{ij}(\gamma
)\right\vert \right\} .
\end{equation*}

\begin{definition}
Let $\mu $ be a $\sigma $-finite Borel measure on $G^{0}$, and let $\pi
\colon C_{c}(G)\rightarrow B(L^{p}(\mu ,\mathcal{Z}))$ be a representation.
We define its \emph{amplification} $\pi ^{(n)}\colon
M_{n}(C_{c}(G))\rightarrow B(\ell ^{p}(n,L^{p}(\mu ,\mathcal{Z})))$ by 
\begin{equation*}
\pi ^{(n)}([f_{ij}]_{i,j\in n})[\xi _{j}]_{j\in n}=\left[ \sum\limits_{j\in
n}\pi (f_{ij})\xi _{j}\right] _{i\in n}\text{.}
\end{equation*}%
The representation $\pi $ is $I$\emph{-norm completely contractive} if $\pi
^{(n)}$ is $I$-norm contractive for every $n\in \mathbb{N}$.
\end{definition}

If one starts with a representation $T$ of a groupoid on a Borel Banach
bundle, one may take its integrated form, and then its amplification to
matrices over $C_{c}(G)$, as in the definition above. The resulting
representation $\pi _{T}^{(n)}$ is the integrated form of a representation
of an amplified groupoid, which we proceed to describe. Given $n\geq 1$,
denote by $G_{n}$ the groupoid $n\times G\times n$ endowed with the product
topology, with set of objects $G^{0}\times n$, and operations defined by%
\begin{equation*}
s(i,\gamma ,j)=(s(\gamma ),j)\ ,\ \ r(i,\gamma ,j)=(r(\gamma ),i)\ \ 
\mbox{
and }\ \ (i,\gamma ,j)(j,\rho ,k)=(i,\gamma \rho ,k)\text{.}
\end{equation*}%
First of all we can observe that we can identify $M_{n}(C_{c}(G))$ with $%
C_{c}(G_{n})$. Furthermore the $I$-norm on $M_{n}(C_{c}(G))$ described above
is exactly the norm on $M_{n}(C_{c}(G))$obtained from the $I$-norm on $%
C_{c}(G_{n})$ via the identification of $M_{n}(C_{c}(G))$ with $C_{c}(G_n)$.

\indent Denote by $\mathcal{Z}^{(n)}$ the Borel Banach bundle over $%
G^{0}\times n$ such that $\mathcal{Z}_{(x,j)}^{(n)}=\mathcal{Z}_{x}$, with
basic sequence $(\sigma _{k}^{(n)})_{k\in \omega }$ defined by $\sigma
_{k,(x,j)}^{(n)}=\sigma _{k,x}$ for $(x,j)\in G^{0}\times n$. Endow $%
G^{0}\times n$ with the measure $\mu ^{(n)}=\mu \times c_{n}$, and define
the \emph{amplification} $T^{(n)}\colon G_{n}\rightarrow \mathrm{Iso}(%
\mathcal{Z}^{(n)})$ of $T$ by $T_{(i,\gamma ,j)}^{(n)}=T_{\gamma }$ for $%
(i,\gamma ,j)\in G_{n}$. It is easy to verify that $\pi _{T}^{(n)}$
corresponds to the integrated form of the representation $T^{(n)}$ under the
canonical identifications of $M_{n}(C_{c}(G))$ with $C_{c}(G_{n})$ and of $%
\ell ^{p}(n,L^{p}(\mu ,\mathcal{Z}))$ with $L^{p}(\mu ^{(n)},\mathcal{Z}%
^{(n)})$. This proves that the representations $\pi _{T}^{(n)}$ and $\pi
_{T^{(n)}}$ are equivalent.

\subsection{Representations of étale groupoids on \texorpdfstring{$L^p$}{Lp}%
-bundles\label{Subsection: representations groupoids Lp-bundles}}

In this section, we want to isolate a particularly important and natural
class of representations of an étale groupoid on Banach spaces. \newline
\indent We fix a quasi-invariant measure $\mu $ on $G^{0}$. Let $\lambda $
be a $\sigma $-finite Borel measure on a standard Borel space $Z$ fibred
over $G^{0}$ via $q$, and assume that $\mu =q_{\ast }(\lambda )$. Denote by $%
\mathcal{Z}$ the $L^{p}$-bundle $\bigsqcup\nolimits_{x\in
G^{0}}L^{p}(\lambda _{x})$ over $\mu $ obtained from the disintegration $%
\lambda =\int \lambda _{x}\ d\mu (x)$ as in Theorem~\ref{thm: BBb structure}.

\begin{definition}
\label{Definition: Lp representation} Adopt the notation from the comments
above. A representation $T\colon G\rightarrow \mathrm{Iso}(\mathcal{Z})$ is
called an \emph{$L^{p}$-representation} of $G$ on $\mathcal{Z}$. Under the
identification $L^{p}(\mu ,\mathcal{Z})\cong L^{p}(\lambda )$ given by
Theorem~\ref{thm: BBb structure}, the integrated form $\pi _{T}\colon
C_{c}(G)\rightarrow B(L^{p}(\lambda ))$ of $T$, is an $I$-norm contractive
nondegenerate representation.
\end{definition}

It will be shown in Theorem~\ref{Theorem: correspondence representations}
that every $I$-norm contractive nondegenerate representation of $C_{c}(G)$
on an $L^{p}$-space is the integrated form of some $L^{p}$-representation of 
$G$.

\begin{remark}
It is clear that an $L^{2}$-representation of $G$ in the sense of Definition %
\ref{Definition: Lp representation}, is a representation of $G$ on a Borel
Hilbert bundle. Conversely, any representation of $G$ on a Borel Hilbert
bundle is equivalent---as in Definition \ref{Definition: equivalence
representations}---to an $L^{2}$-representation. In fact, if $\mathcal{H}$
is a Borel Hilbert bundle over $G^{0}$, then for every $0\leq \alpha \leq
\omega $ the set $X_{\alpha }=\{x\in G^{0}\colon \mathrm{dim}(\mathcal{H}%
_{x})=\alpha \}$ is Borel. Thus, $\mathcal{H}$ is isomorphic to the Hilbert
bundle $\mathcal{Z}_{0}=\bigsqcup\nolimits_{0\leq \alpha \leq \omega
}X_{\alpha }\times \ell ^{2}(\alpha )$.Set $Z=\bigsqcup\nolimits_{0\leq
\alpha \leq \omega }(X_{\alpha }\times \alpha )$, and define a $\sigma $%
-finite Borel measure $\lambda $ on $Z$ by $\lambda
=\bigsqcup\nolimits_{0\leq \alpha \leq \omega }(\mu \times c_{\alpha })$. It
is immediate that $\mathcal{Z}_{0}$ is (isomorphic to) the Borel Hilbert
bundle $\bigsqcup\nolimits_{x\in G^{0}}L^{2}(\lambda _{x})$ obtained from
the disintegration of $\lambda $ with respect to $\mu $.
\end{remark}

In view of the above remark, there is no difference, up to equivalence,
between $L^{2}$-representations and representations on Borel Hilbert
bundles. The theory of $L^{p}$-representations of $G$ for $p\in (1,\infty) $
can therefore be thought of as a generalization of the theory of
representations of $G$ on Borel Hilbert bundles.

\begin{example}[Left regular representation]
\label{Example: left regular representation} Take $Z=G$ and $\lambda =\nu $,
in which case the disintegration of $\lambda $ with respect to $\mu $ is $%
(c_{xG})_{x\in X}$. For $\gamma \in G$, define the surjective linear
isometry 
\begin{equation*}
T_{\gamma }^{\mu ,p}\colon \ell ^{p}(s(\gamma )G)\rightarrow \ell
^{p}(r(\gamma )G)
\end{equation*}%
by $(T_{\gamma }^{\mu ,p}\xi )(\rho )=\xi (\gamma ^{-1}\rho )$. The
assignment $\gamma \mapsto T_{\gamma }^{\mu ,p}$ defines a representation $%
T^{\mu ,p}$ of $G$ on the Borel Banach bundle $\bigsqcup\nolimits_{x\in
G^{0}}\ell ^{p}(xG)$ which we shall call the \emph{left regular }$L^{p}$-%
\emph{representation} of $G$ associated with $\mu $.
\end{example}

When the Hölder exponent $p$ is clear from the context, we will write $%
T^{\mu }$ in place of $T^{\mu ,p}$. A straightforward computation shows that
the dual $(T^{\mu ,p})^{\prime }$ of the left regular $L^{p}$-representation
associated with $\mu $ is the left regular $L^{p^{\prime }}$-representation $%
T^{\mu ,p^{\prime }}$ associated with $\mu $. Following Rieffel's induction
theory and for consistency with \cite[Section 3.1 and Appendix D]%
{paterson_groupoids_1999} we denote by $\mathrm{Ind}\left( \mu \right) $ the
integrated form of $T^{\mu ,p}$. The same computation as in \cite[page 100]%
{paterson_groupoids_1999} where $2$ is replaced by $p$ shows that $\mathrm{%
Ind}(\mu )$ is the left action of $C_{c}(G)$ on $L^{p}(\nu ^{-1})$ by
convolution. This allows one to deduce that in the Hausdorff case the left
regular representations of $G$ separate points. The non-Hausdorff case is
more subtle. A treatment of left-regular representations of non-Hausdorff
groupoids on Hilbert spaces is given in \cite{khoshkam_regular_2002}.

\begin{definition}
Let us say that a family $\mathcal{M}$ of quasi-invariant probability
measures on $G^{0}$ \emph{separates points}, if for every nonzero function $%
f\in C_{c}(G)$, there is a measure $\mu \in \mathcal{M}$ such that $f$ does
not vanish on the support of the integrated measure $\nu =\int c_{xG}\ d\mu
(x)$. Similarly, a collection $\mathcal{R}$ of representations of $C_{c}(G)$
on Banach algebras is said to \emph{separate points} if for every nonzero
function $f\in C_{c}(G)$, there is a representation $\pi \in \mathcal{R}$
such that $\pi (f)$ is nonzero.
\end{definition}

\begin{lemma}
\label{Lemma: faithful measure} If $G$ is a Hausdorff étale groupoid, then a
function $f$ in $C_{c}(G)$ belongs to $\mathrm{Ker}(\mathrm{Ind}(\mu ))$ if
and only if it vanishes on the support of $\nu $.
\end{lemma}

\begin{proof}
Suppose that $f$ vanishes on the support of $\nu $. Then 
\begin{equation*}
\left\langle \mathrm{Ind}(\mu )(f)\xi ,\eta \right\rangle _{L^{p}(\nu
)}=\int f(\gamma )\left\langle T_{\gamma }\xi _{s(\gamma )},\eta _{r(\gamma
)}\right\rangle D^{-\frac{1}{p}}(\gamma )\ d\nu (\gamma )=0\text{.}
\end{equation*}%
for every $\xi \in L^{p}(\nu )$ and $\eta \in L^{p^{\prime }}(\nu )$, so $%
\mathrm{Ind}(\mu )(f)=0$. Conversely, if $\mathrm{Ind}(\mu )(f)=0$ then $%
f\ast \xi =0$ for every $\xi \in L^{p}(\nu ^{-1})$. In particular, $f=f\ast
\chi _{G^{0}}=0$ in $L^{p}(\nu ^{-1})$. Thus $f(\gamma )=0$ for $\nu ^{-1}$%
-almost every $\gamma \in G$ and hence also for $\nu $-almost every $\gamma
\in G$. Since $G$ is assumed to be Hausdorff, $f$ is continuous, and hence
it vanishes on the support of $\nu $.
\end{proof}

By Lemma~\ref{Lemma: faithful measure}, in the Hausdorff case a family $%
\mathcal{M}$ of Borel probability measures on $G^{0}$ separates points if
and only if the collection of left regular representations associated with
elements of $\mathcal{M}$ separates points.

\begin{proposition}
\label{Proposition: separates points}If $G$ is a Hausdorff étale groupoid,
then the family of left regular representations associated with
quasi-invariant Borel probability measures on $G^{0}$ separates points.
\end{proposition}

\begin{proof}
A quasi-invariant Borel probability measure is said to be transitive if it
is supported by an orbit. Every orbit carries a transitive measure, which is
unique up to equivalence; see \cite[Definition 3.9 of Chapter 1 ]%
{renault_groupoid_1980}. It is well known that the transitive measures
constitute a collection of quasi-invariant Borel probability measures on $%
G^{0}$ that separates points; see \cite[Proposition 1.11 of Chapter 2]%
{renault_groupoid_1980}, so the proof is complete.
\end{proof}

\section{Representations of inverse semigroups on \texorpdfstring{$L^p$}{Lp}%
-spaces}

\subsection{The Banach-Lamperti theorem\label{Subsection: Banach-Lamperti
theorem}}

Let $\mu $ and $\nu $ be Borel probability measures on standard Borel spaces 
$X$ and $Y$, and let $p\in \lbrack 1,\infty )\setminus \left\{ 2\right\} $.
The \emph{Banach-Lamperti theorem }\cite[Theorem 3.2.5]%
{fleming_isometries_2003} classifies the linear isometries from $L^{p}\left(
\mu \right) $ to $L^{p}\left( \nu \right) $ in terms of \emph{regular set
isomorphisms}; see \cite[Definition 3.2.3]{fleming_isometries_2003}. We are
interested in the case of \emph{surjective }linear isometries. In such a
case the Banach-Lamperti theorem can be stated as follows.

\begin{theorem}[Banach-Lamperti]
\label{Theorem: Banach-Lamperti} Let $p\in \lbrack 1,\infty )\setminus \{2\}$%
. If $u:L^{p}(\mu )\rightarrow L^{p}(\nu )$ is a surjective linear isometry,
then there are conull Borel subsets $X_{0}$ and $Y_{0}$ of $X$ and $Y$, a
Borel isomorphism $\phi \colon X_{0}\rightarrow Y_{0}$ such that $\phi
_{\ast }(\mu )|_{X_{0}}\sim \nu |_{Y_{0}}$, and a Borel function $g\colon
Y\rightarrow \mathbb{C}$ with $\left\vert g(y)\right\vert ^{p}=\frac{d\phi
_{\ast }(\mu )}{\ d\nu }(y)$ for $\nu $-almost every $y\in Y$, such that $%
u\xi =g\cdot \left( \xi \circ \phi ^{-1}\right) $ for every $\xi \in
L^{p}(\nu )$.
\end{theorem}

Theorem \ref{Theorem: Banach-Lamperti} was proved by Banach in \cite[Section
5]{banach_theorie_1993} and later generalized by Lamperti to not necessarily
surjective isometries \cite{lamperti_isometries_1958}.

\subsection{Hermitian idempotents and spatial partial isometries}

Let $X$ be a complex vector space. The following definition is taken from 
\cite{lumer_semi-inner-product_1961}.

\begin{definition}
A \emph{semi-inner product} on $X$ is a function $[ \cdot ,\cdot ]\colon
X\times X\rightarrow \mathbb{C}$ satisfying:

\begin{enumerate}
\item $\left[ \cdot ,\cdot \right] $ is linear in the first variable;

\item $\left[ x,\lambda y\right] =\overline{\lambda }\left[ x,y\right] $ for
every $\lambda \in \mathbb{C}$ and $x,y\in X$;

\item $\left[ x,x\right] \geq 0$ for every $x\in X$, and equality holds if
and only if $x=0$;

\item $\left\vert \left[ x,y\right] \right\vert \leq \left[ x,x\right] ^{%
\frac{1}{2}}\left[ y,y\right] ^{\frac{1}{2}}$ for every $x,y\in X$.
\end{enumerate}

The \emph{norm} on $X$ associated with the semi-inner product $[ \cdot
,\cdot ] $ is defined by $\left\| x\right\| =\left[ \cdot ,\cdot \right] ^{%
\frac{1}{2}}$ for $x\in X$.
\end{definition}

In general, there might be different semi-inner products on $X$ that induce
the same norm. Nonetheless, it is not difficult to see that on a smooth
Banach space---and in particular on $L^{p}$-spaces---there is at most one
semi-inner product compatible with its norm; see the remark after the proof
of Theorem~3 in \cite{lumer_semi-inner-product_1961}.

\begin{definition}
A semi-inner product on a Banach space that induces its norm is called \emph{%
compatible}. A Banach space $X$ endowed with a compatible semi-inner product
is called a \emph{semi-inner product space}.
\end{definition}

By the above discussion, if $X$ is a smooth Banach space, then a compatible
semi-inner product---when it exists---is unique.

\begin{remark}
It is easy to verify that the norm of $L^{p}( \lambda) $ is induced by the
semi-inner product%
\begin{equation*}
\left[ f,g\right] =\left\| g\right\| _{p}^{2-p}\int f(x) \overline{g(x)}%
\left\vert g(x) \right\vert ^{p-2} d\lambda(x)
\end{equation*}%
for $f,g\in L^{p}(\lambda) $ with $g\neq 0$.
\end{remark}

An \emph{inner product }on $X$ is precisely a semi-inner product such that
moreover $\left[ x,y\right] =\overline{\left[ y,x\right] }$ for every $%
x,y\in X$. Semi-inner products allow one to generalize notions concerning
operators on Hilbert spaces to more general Banach spaces.

\begin{definition}
Let $X$ be a semi-inner product space, and let $T\in B(X)$. The \emph{%
numerical range} $W( T) $ of $T$, is the set 
\begin{equation*}
\left\{[ Tx,x]\colon x\in X, [ x,x] =1\right\}\subseteq \mathbb{C}.
\end{equation*}
The operator $T$ is called \emph{hermitian} if $W( T) \subseteq \mathbb{R}$.
\end{definition}

Adopt the notation and terminology from the definition above. In view of 
\cite{lumer_semi-inner-product_1961} the following statements are equivalent:

\begin{enumerate}
\item $T$ is hermitian;

\item $\left\| 1+irT\right\| =1+o\left( r\right) $ for $r\rightarrow 0$;

\item $\left\| \exp \left( irT\right) \right\| =1$ for all $r\in \mathbb{R} $%
.
\end{enumerate}

It is clear that when $X$ is a Hilbert space, an operator is hermitian if
and only if it is self-adjoint. In particular, the hermitian idempotents on
a Hilbert space are exactly the orthogonal projections.

Let $\lambda $ be a Borel measure on a standard Borel space $Z$. Hermitian
idempotents on $L^{p}(\lambda )$, for $p\neq 2$, have been characterized in 
\cite[Chapter 3]{torrance_adjoints_1968} (see also \cite%
{berkson_hermitian_1972}): these are precisely the multiplication operators
associated with characteristic functions on Borel subsets of $Z$.

Recall that a bounded linear operator $s$ on a Hilbert space is a \emph{%
partial isometry }if there is another bounded linear operator $t$ such that $%
st$ and $ts$ are orthogonal projections. The following is a generalization
of partial isometries on Hilbert spaces to $L^p$-spaces. We use the term
`spatial' in accordance to the terminology in \cite%
{phillips_analogs_2012,phillips_simplicity_2013,phillips_isomorphism_2013,phillips_crossed_2013}%
.

\begin{definition}
\label{definition: spatial}Let $X$ be a semi-inner product space and $s\in
B(( X) $. We say that $s$ is a \emph{partial isometry} if $(\| s\| \leq 1$
and there exists $t\in B(( X) $ such that $(\| t\| \leq 1$ and $st$ and $ts$
are idempotent. If moreover $st$ and $ts$ are \emph{hermitian }idempotents
then we say that $s$ is \emph{spatial} partial isometry.
\end{definition}

Following \cite{phillips_analogs_2012}, we call an element $t$ as in
Definition \ref{definition: spatial} a \emph{reverse} of $s$. We call $ts$
and $st$ the \emph{source} and \emph{range} idempotents of $s$,
respectively. We denote by $\mathcal{S}(X)$ the set of all spatial partial
isometries in $B(X)$, and by $\mathcal{E}(X)$ the set of hermitian
idempotents in $B(X)$.

It is a standard fact in Hilbert space theory that all partial isometries on
a Hilbert space are spatial. Moreover, the reverse of a partial isometry on
a Hilbert space is unique, and it is given by its adjoint. The situation for
partial isometries on $L^{p}$-spaces, for $p\neq 2$, is rather different.
The following proposition can be taken as a justification for the term
\textquotedblleft spatial\textquotedblright .

\begin{proposition}
Let $p\in (1,\infty )\setminus \{2\}$ and let $\lambda $ be a $\sigma $%
-finite Borel measure on a standard Borel space $Z$. If $e$ is a hermitian
idempotent in $L^{p}(\lambda )$, then there is a Borel subset $E$ of $Z$
such that $e=\Delta _{\chi _{E}}$. More generally, if $s$ is a spatial
partial isometry on $L^{p}(\lambda )$, then there are Borel subsets $E$ and $%
F$ of $Z$, a Borel isomorphism $\phi \colon E\rightarrow F$, and a Borel
function $g\colon F\rightarrow \mathbb{C}$ such that 
\begin{equation*}
(s\xi )(y)=%
\begin{cases}
g(y)\cdot (\xi \circ \phi ^{-1})(y) & \text{if }y\in F\text{, and} \\ 
0 & \text{ otherwise}%
\end{cases}%
\end{equation*}%
for all $\xi $ in $L^{p}(\lambda )$ and for $\lambda $-almost every $y\in Z$.
\end{proposition}

\begin{proof}
The result follows from the characterization of hermitian idempotents
mentioned above, together with the Banach-Lamperti theorem.
\end{proof}

\begin{remark}
Adopt the notation of the above proposition. It is easy to check that the
reverse of $s$ is also spatial, and that it is given by 
\begin{equation*}
(t\xi )(y)=%
\begin{cases}
\overline{(g\circ \phi )(y)}\cdot (\xi \circ \phi )(y) & \text{if }y\in E%
\text{, and} \\ 
0 & \text{ otherwise}%
\end{cases}%
\end{equation*}%
for all $\xi $ in $L^{p}(\lambda )$ and for $\lambda $-almost every $y\in Z$%
. In particular, the reverse of a spatial partial isometry of an $L^{p}$%
-space is unique. We will consequently write $s^{\ast }$ for the reverse of
a spatial partial isometry $s$.
\end{remark}

For $p\in \left( 1,+\infty \right) \setminus \left\{ 2\right\} $ the set $%
\mathcal{S}(L^{p}(\lambda ))$ of spatial partial isometries on $%
L^{p}(\lambda )$ is an inverse semigroup, and the set $\mathcal{E}%
(L^{p}(\lambda ))$ of hermitian idempotents on $L^{p}(\lambda )$ is
precisely the semilattice of idempotent elements of $\mathcal{S}%
(L^{p}(\lambda ))$. Moreover, the map $\mathcal{B}_{\lambda }\rightarrow 
\mathcal{E}(L^{p}(\lambda ))$ given by $F\mapsto \Delta _{\chi _{F}}$ is an
isomorphism of semilattices. Thus, $\mathcal{E}(L^{p}(\lambda ))$ is a
complete Boolean algebra.

\begin{remark}
If $(e_{j})_{j\in I}$ is an increasing net of hermitian idempotents, then $%
\sup\nolimits_{j\in I}e_{j}$ is the limit of the sequence $(e_{j})_{j\in I}$
in the strong operator topology.
\end{remark}

\subsection{Representations of inverse semigroups}

We now turn to inverse semigroup representations on $L^{p}$-spaces by
spatial partial isometries. Fix an inverse semigroup $\Sigma $, and recall
that $M_{n}(\Sigma )$ has a natural structure of inverse semigroup for every 
$n\geq 1$ by \cite[Proposition 2.1.4]{paterson_groupoids_1999}.

\begin{definition}
Let $\lambda$ be a $\sigma $-finite Borel measure on a standard Borel space.
A \emph{representation} of $\Sigma $ on $L^{p}(\lambda) $ is a semigroup
homomorphism $\rho \colon \Sigma \to \mathcal{S}( L^{p}(\lambda))$.

For $n\geq 1$ denote by $\lambda ^{\left( n\right) }$ the measure $\lambda
\times c_{n}$, where $c_{n}$ is the counting measure on $n$. We define the 
\emph{amplification} $\rho ^{(n)}\colon M_{n}(\Sigma )\rightarrow \mathcal{S}%
(L^{p}(\lambda ^{(n)}))$ of $\rho $, by $\rho ^{n}([\sigma _{ij}]_{i,j\in
n})=[\rho (\sigma _{ij})]_{i,j\in n}$, where we identify $B(L^{p}(\lambda
^{(n)}))$ with $M_{n}(B(L^{p}(\lambda )))$ in the usual way.\newline
\indent The \emph{dual }of $\rho $ is the representation $\rho ^{\prime
}\colon \Sigma \rightarrow \mathcal{S}(L^{p^{\prime }}(\lambda ))$ given by $%
\rho ^{\prime }(\sigma )=\rho (\sigma ^{\ast })^{\prime }$ for $\sigma \in
\Sigma $.
\end{definition}

\begin{definition}
\label{Definition:CSIGMA}Denote by $\mathbb{C}\Sigma $ the complex *-algebra
of formal linear combinations of elements of $\Sigma $, with operations
determined by $\delta _{\sigma }\delta _{\tau }=\delta _{\sigma \tau }$ and $%
\delta _{\sigma }^{\ast }=\delta _{\sigma ^{\ast }}$ for all $\sigma ,\tau
\in \Sigma $, and endowed with the $\ell ^{1}$-norm. The canonical
identification of $\mathbb{C}M_{n}(\Sigma )$ with $M_{n}(\mathbb{C}\Sigma )$
for $n\geq 1$, defines matrix norms on $\mathbb{C}\Sigma $.
\end{definition}

\begin{remark}
Every representation $\rho \colon \Sigma \rightarrow \mathcal{S}%
(L^{p}(\lambda ))$ induces a contractive representation $\pi _{\rho }\colon 
\mathbb{C}\Sigma \rightarrow B(L^{p}(\lambda ))$ such that $\pi _{\rho
}(\delta _{\sigma })=\rho \left( \sigma \right) $ for $\sigma \in \Sigma $.
It is not difficult to verify the following facts:

\begin{enumerate}
\item Since, for $n\geq 1$, the amplification $\pi _{\rho }^{(n)}$ of $\pi
_{\rho }$ to $M_{n}(\mathbb{C}\Sigma) $ is the representation associated
with the amplification $\rho ^{(n) }$ of $\rho $, it follows that $\pi
_{\rho }$ is $p$-completely contractive.

\item The representation $\pi _{\rho ^{\prime }}$ associated with the dual $%
\rho ^{\prime }$ of $\rho $ is the dual of the representation $\pi _{\rho }$
associated with $\rho $.
\end{enumerate}
\end{remark}

\begin{definition}
Let $\lambda $ and $\mu $ be $\sigma $-finite Borel measure on standard
Borel spaces, and let $\rho $ and $\kappa $ be representations of $\Sigma $
on $L^{p}(\lambda) $ and $L^{p}(\mu) $ respectively. We say that $\rho $ and 
$\kappa $ are \emph{equivalent }if there is a surjective linear isometry $%
u\colon L^{p}(\lambda) \to L^{p}(\mu) $ such that $u\rho (\sigma) =\kappa
(\sigma) u$ for every $\sigma \in \Sigma $.
\end{definition}

Adopt the notation of the definition above. If $\rho $ and $\kappa $ are
equivalent, then their dual representations $\rho ^{\prime }$ and $\kappa
^{\prime }$ are also equivalent. Similarly, if $\rho $ and $\kappa $ are
equivalent, then the corresponding representations $\pi _{\rho }$ and $\pi
_{\kappa }$ of $\mathbb{C}\Sigma $ are equivalent.

\subsection{Tight representations of semilattices}

In the following, all semilattices will be assumed to have a minimum element 
$0$. Consistently, all inverse semigroups will be assumed to have a neutral
element $0$, which is the minimum of the associated idempotent semilattice.
In the rest of this subsection we recall some definitions from Section 11 of 
\cite{exel_inverse_2008}.

\begin{definition}
\label{Definition:cover}Let $E$ be a semilattice and let $\mathcal{B}=(%
\mathcal{B},0,1,\wedge ,\vee ,\lnot )$ be a Boolean algebra. A \emph{%
representation} of $E$ on $\mathcal{B}$ is a semilattice morphism $%
E\rightarrow (\mathcal{B},\wedge )$ satisfying $\beta (0)=0$.

Two elements $x,y$ of $E$ are said to be \emph{orthogonal}, written $x\perp
y $, if $x\wedge y=0$. Furthermore, we say that $x$ and $y$ \emph{intersect
(each other)} if they are not orthogonal. \newline
\indent If $X\subseteq Y\subseteq E$, then $X$ is a \emph{cover }for $Y$ if
every nonzero element of $Y$ intersects an element of $X$.
\end{definition}

It is easy to verify that a representation of a semilattice $E$ on a Boolean
algebra sends orthogonal elements to orthogonal elements. It is also
immediate to check that a cover for the set of predecessors of some $x\in E$
is also a cover for $\{x\}$.

\begin{notation}
If $X$ and $Y$ are (possibly empty) subsets of $E$, we denote by $E^{X,Y}$
the set 
\begin{equation*}
E^{X,Y}=\{z\in E\colon z\leq x\ \mbox{ for all }x\in X,\mbox{ and }z\perp y\ %
\mbox{ for all }y\in Y\}.
\end{equation*}
\end{notation}

We are now ready to state the definition of tight representation of a
semilattice.

\begin{definition}
\label{Definition: tight representation semilattice} Let $E$ be a
semilattice and let $\mathcal{B}$ be a Boolean algebra. A representation $%
\beta \colon E\rightarrow \mathcal{B}$ is said to be \emph{tight} if for
every pair $X,Y$ of (possibly empty) finite subsets of $E$ and every finite
cover $Z$ of $E^{X,Y}$, we have 
\begin{equation}
\bigvee\limits_{z\in Z}\beta (z)=\bigwedge_{x\in X}\beta (x)\wedge
\bigwedge_{y\in Y}\lnot \beta (y).\text{\label{Equation: tight}}
\end{equation}
\end{definition}

Proposition 11.9 of \cite[Proposition 11.9]{exel_inverse_2008} shows that,
when $E$ is also a Boolean algebra, the tight representations of $E$ are
precisely the Boolean algebra homomorphisms.

\begin{definition}
\label{Definition:dense-subsemilattice}Suppose that $E$ is a semilattice. A
subsemilattice $F$ of $E$ is \emph{dense} if for every $x\in E$ nonzero
there is $y\in F$ nonzero such that $y\leq x$.
\end{definition}

\begin{lemma}
\label{Lemma: dense}Suppose that $E$ is a semilattice, and $F$ is a dense
subsemilattice of $E$. If $\beta $ is a tight representation of $E$ on a
Boolean algebra $\mathcal{B}$, then the restriction of $\beta $ to $F$ is tight.
\end{lemma}

\begin{proof}
Suppose that $X,Y\subset F$ and that $Z$ is a cover for $F^{X,Y}$. We claim
that $Z$ is a cover for $E^{X,Y}$. Pick $x\in E^{X,Y}$ nonzero, and use 
density of $F$ in $E$ to find a nonzero $y\in F$ 
such that $y\leq x$. Since $y\in F^{X,Y}$ and $Z$ is a
cover for $F^{X,Y}$, there is $z\in Z$ such that $z$ and $y$ intersect.
Therefore also $z$ and $x$ intersect. This shows that $Z$ is a cover for $%
E^{X,Y}$. Therefore Equation~\eqref{Equation: tight} holds, as desired.
\end{proof}

\subsection{Tight representations of inverse semigroups on $L^p$-spaces}

As in the case of representation of inverse semigroups on Hilbert spaces
(see \cite[Section 13]{exel_inverse_2008}), we will isolate a class of
\textquotedblleft well behaved\textquotedblright\ representations of inverse
semigroups on $L^{p}$-spaces. The following definition is a natural
generalization of \cite[Definition 13.1]{exel_inverse_2008}.

\begin{definition}
\label{Definition: tight representation inverse semigroup} Let $\lambda $ be
a $\sigma $-finite Borel measure on a standard Borel space. A representation 
$\rho \colon \Sigma \rightarrow \mathcal{S}(L^{p}(\lambda ))$ is said to be 
\emph{tight }if its restriction to the idempotent semilattice $E(\Sigma )$
of $\Sigma $ is a tight representation on the Boolean algebra $\mathcal{E}%
(L^{p}(\lambda ))$ of hermitian idempotents.
\end{definition}

\begin{remark}
If $\rho \colon \Sigma \rightarrow \mathcal{S}(L^{p}(\lambda ))$ is a tight
representation as above, then the net $(\rho (\sigma ))_{\sigma \in E(\Sigma
)}$ converges to the identity in the strong operator topology. Thus,
tightness should be thought of as a \emph{nondegeneracy} condition for
representations of inverse semigroups.
\end{remark}

\begin{definition}
\label{Definition: regular tight} A tight representation $\rho $ of $\Sigma $
on $L^{p}(\lambda )$ is said to be \emph{regular} if, for every idempotent
open bisection $U$ of $G$, the element $\rho ( U) $ is the limit of the net $%
( \rho ( V) ) _{V}$ where $V$ ranges among all idempotent open bisections of 
$G$ with compact closure contained in $U$, ordered by inclusion. In formulas,%
\begin{equation}
\rho (U)=\lim_{V\in E(\Sigma _{c}(G)),\overline{V}\subseteq U}\rho (V)\text{%
\label{Equation: regular relation}.}
\end{equation}
\end{definition}

\subsection{Representations of semigroups of bisections\label{Subsection:
integrated representation slices}}

Let $G$ be an étale groupoid, let $\lambda $ be a $\sigma $-finite Borel
measure on a standard Borel space, and let $\pi $ be a contractive
nondegenerate representation of $C_{c}(G)$ on $L^{p}(\lambda )$. Denote by $%
\Sigma _{c}(G)$ the inverse semigroup of precompact open bisections of $G$.
In this subsection, we show how to associate to $\pi $ a tight, regular
representation $\rho _{\pi }$ of $\Sigma _{c}(G)$ on $L^{p}(\lambda )$.

Given a precompact open bisection $A$ of $G$, $\xi \in L^{p}(\lambda )$, and 
$\eta \in L^{p^{\prime }}(\lambda )$, the assignment $f\mapsto \langle \pi
(f)\xi ,\eta \rangle $ is a $\| \cdot \| _{\infty }$-continuous linear
functional on $C_{c}(A)$ of norm at most $\| \xi \| \| \eta \| $. By the
Riesz-Markov-Kakutani representation theorem, there is a Borel measure $\mu
_{A,\xi ,\eta }$ supported on $A$, of total mass at most $\| \xi \| \| \eta
\| $, such that%
\begin{equation}
\langle \pi (f)\xi ,\eta \rangle =\int f\text{ }d\mu _{A,\xi ,\eta }\text{%
\label{Equation: integral}}
\end{equation}%
for every $f\in C_{c}(G)$. If $A,B\in \Sigma _{c}(G)$, then $\mu _{A,\xi
,\eta }$ and $\mu _{B,\xi ,\eta }$ coincide on $A\cap B$. Arguing as in \cite%
[page 87, and pages 98-99]{paterson_groupoids_1999}, we conclude that there
is a Borel measure $\mu _{\xi ,\eta }$ defined on all of $G$, such that $\mu
_{A,\xi ,\eta }$ is the restriction of $\mu _{\xi ,\eta }$ to $A$, for every 
$A\in \Sigma _{c}(G)$, and moreover $\langle \pi (f)\xi ,\eta \rangle =\int
f\ d\mu _{\xi ,\eta }$ for every $f\in C_{c}(G)$.

\begin{lemma}
\label{Lemma: dense span} The linear span of $\{ \pi(\chi_A) \xi \colon A\in
\Sigma _{c}(G), \xi \in L^{p}(\lambda)\}$ is dense in $L^p(\lambda)$.
\end{lemma}

\begin{proof}
Let $\eta \in L^{p^{\prime }}(\lambda )$ satisfy $\langle \rho _{\pi }(A)\xi
,\eta \rangle =0$ for every $\xi \in L^{p}(\lambda )$ and every $A\in \Sigma
_{c}(G)$. Since $\rho_{\pi}(A)=\pi(\chi_A)$, we have $\langle \rho _{\pi
}(A)\xi ,\eta \rangle =\int \chi _{A}\ d\mu _{\xi ,\eta }$, so it follows
that $\mu _{\xi ,\eta }(A)=0$ for every $\xi \in L^{p}(\lambda )$ and every $%
A\in \Sigma _{c}(G)$. Thus $\langle \pi (f)\xi ,\eta \rangle =0$ for every $%
f\in C_{c}(G)$ and every $\xi \in L^{p}(\lambda )$. Since $\pi $ is
nondegenerate, we conclude that $\eta =0$, which finishes the proof.
\end{proof}

Equation~\eqref{Equation: integral} allows one to extend $\pi $ to $B_{c}(G)$
by defining%
\begin{equation*}
\left\langle \pi (f)\xi ,\eta \right\rangle =\int fd\mu _{\xi ,\eta }
\end{equation*}%
for $f\in B_{c}(G)$, $\xi \in L^{p}(\lambda )$, and $\eta \in L^{p}\left(
\eta \right) $. Lemma~2.2.1 of \cite{paterson_groupoids_1999} shows that $%
\pi $ is indeed a nondegenerate representation of $B_{c}(G)$ on $%
L^{p}(\lambda )$. In particular the function$\ \rho _{\pi }:A\mapsto \pi
(\chi _{A})$ is a semigroup homomorphism from $\Sigma _{c}(G)$ to $B\left(
L^{p}(\lambda )\right) $. We will show below that such a function is a
tight, regular representation of $\Sigma _{c}(G)$ on $L^{p}(\lambda )$

Suppose that $f\in B(G^{0})$. Define $\pi (f)\in B\left( L^{p}(\lambda
)\right) $ by 
\begin{equation}
\left\langle \pi (f)\xi ,\eta \right\rangle =\int fd\mu _{\xi ,\eta }\text{%
\label{Equation: integral2}}
\end{equation}%
for $\xi \in L^{p}(\lambda )$, and $\eta \in L^{p^{\prime }}(\lambda )$.
Since 
\begin{equation}
\pi (fg)=\pi (f)\pi (g)\text{\label{Equation: multiplicative}}
\end{equation}%
for $f,g\in B_{c}(G^{0})$, it follows via a monotone classes argument that
Equation~\eqref{Equation: multiplicative} holds for any $f,g\in B(G^{0})$.
In particular, $\pi (\chi _{A})$ is an idempotent for every $A\in \mathcal{B}%
(G^{(0)})$. It follows from Lemma~\ref{Lemma: dense span} that $\pi \left(
\chi _{G^{0}}\right) $ is the identity operator on $L^{p}(\lambda )$. Fix
now $A\in \mathcal{B}(G^{0})$ and $r\in \mathbb{R}$. For any $\xi \in
L^{p}(\lambda )$ and $\eta \in L^{p^{\prime }}(\lambda )$ such that $%
\left\Vert \xi \right\Vert ,\left\Vert \eta \right\Vert \leq 1$ we have that%
\begin{equation*}
\left\vert \left\langle \left( 1+ir\pi (\chi _{A})\right) \xi ,\eta
\right\rangle \right\vert =\left\vert \int \left( \chi _{G^{0}}+ir\chi
_{A}\right) d\mu _{\xi ,\eta }\right\vert \leq \left\Vert \chi
_{G^{0}}+ir\chi _{A}\right\Vert _{\infty }\leq 1+\frac{1}{2}r^{2}\text{.}
\end{equation*}%
Therefore $\left\Vert 1+ir\pi (\chi _{A})\right\Vert \leq 1+\frac{1}{2}r^{2}$%
. This shows that $\pi (\chi _{A})$ is an hermitian idempotent of $%
L^{p}(\lambda )$. It follows from Equation~\eqref{Equation: integral2} and
Equation~\eqref{Equation: multiplicative} that the function $A\mapsto \pi
(\chi _{A})$ is a $\sigma $-complete homomorphism of Boolean algebras from $%
\mathcal{B}(G^{0})$ to $\mathcal{E}\left( L^{p}(\lambda )\right) $. In
particular $A\rightarrow \pi \left( \chi _{A}\right) $ a tight semilattice
representation; see \cite[Proposition 11.9]{exel_inverse_2008}. By Lemma \ref%
{Lemma: dense} the restriction of a tight representation to a dense
subsemilattice---in the sense of Definition \ref%
{Definition:dense-subsemilattice}---is still tight, it follows that the
function $\rho _{\pi }:A\mapsto \pi (\chi _{A})$ for $A\in \Sigma _{c}(G)$
is a tight representation of $\Sigma _{c}(G)$ on $L^{p}(\lambda )$, which is
moreover regular by $\sigma $-completeness. The same argument shows that if $%
\Sigma $ is an inverse subsemigroup of $\Sigma _{c}(G)$ which forms a basis
for the topology of $G$, then the restriction of $\rho _{\pi }$ to $\Sigma $
is a tight, regular representation of $\Sigma $ on $L^{p}(\lambda )$.

\begin{remark}
It is clear that if $\pi $ and $\widetilde{\pi }$ are $I$-norm contractive
nondegenerate representations of $C_{c}(G)$ on $L^{p}$-spaces, then $\pi $
and $\widetilde{\pi }$ are equivalent if and only if $\rho _{\pi }$ and $%
\rho _{\widetilde{\pi }}$ are equivalent.
\end{remark}

\section{Disintegration of representations}

Throughout this section, we let $G$ be an étale groupoid and $\Sigma $ be an
inverse subsemigroup of $\Sigma _{c}(G)$ that forms a basis for the topology
of $G$. Let $\lambda $ be a $\sigma $-finite measure on a standard Borel
space $Z$.

\subsection{The disintegration theorem}

\begin{theorem}
\label{Theorem: disintegration representation inverse} If $\rho \colon
\Sigma \rightarrow \mathcal{S}(L^{p}(\lambda ))$ is a tight, regular
representation, then there exist a Borel map $q:Z\rightarrow G^{0}$ such
that $\mu =q_{\ast }(\lambda )$ is a quasi-invariant Borel measure on $G^{0}$%
. Furthermore if $\lambda =\int \lambda _{x}\ d\mu (x)$ denotes the
disintegration of $\lambda $ with respect to $\mu $, then there exists a
representation $T$ of $G$ on the Borel Banach bundle $\bigsqcup\nolimits_{x%
\in G^{0}}L^{p}(\lambda _{x})$, such that 
\begin{equation*}
\langle \rho (A)\xi ,\eta \rangle =\int_{A}D(\gamma )^{-\frac{1}{p}}\langle
T_{\gamma }\xi _{s(\gamma )},\eta _{r(\gamma )}\rangle \ d\nu (\gamma )
\end{equation*}%
for $A\in \Sigma $, for $\xi \in L^{p}(\lambda )$, and for $\eta \in
L^{p^{\prime }}(\lambda )$.
\end{theorem}

The rest of this section is dedicated to the proof of the theorem above. For
simplicity and without loss of generality, we will focus on the case where $%
\lambda $ is a probability measure. In the following, we fix a
representation $\rho $ as in the statement of Theorem~\ref{Theorem:
disintegration representation inverse}.

\subsection{Fibration}

Define $\Phi \colon E(\Sigma )\rightarrow \mathcal{B}_{\lambda }$ by $\Delta
_{\chi _{\Phi (U)}}=\rho (U)$ for $U\in E(\Sigma )$. Denote by $\mathcal{U}$
the semilattice of open subsets of $G^{0}$. Extend $\Phi $ to a function $%
\mathcal{U}\rightarrow \mathcal{B}_{\lambda }$ by setting 
\begin{equation*}
\Phi (V)=\bigcup_{W\in E(\Sigma _{c}),\overline{W}\subseteq U}\Phi (W).
\end{equation*}%
Then $\Delta _{\chi _{\Phi (V)}}$ is the limit in the strong operator
topology of the increasing net $(\Delta _{\chi _{\rho (W)}})_{W\in E(\Sigma
_{c}),\overline{W}\subseteq V}$. By Equation~ 
\eqref{Equation:
regular relation}, the expression above indeed defines an extension of $\Phi 
$. Moreover, a monotone classes argument shows that $\Phi $ is a
representation. Tightness of $\rho $ together with Equation~ 
\eqref{Equation:
regular relation} further imply that $\Phi (U\cup V)=\Phi (U)\cup \Phi (V)$
whenever $U$ and $V$ are disjoint, and that%
\begin{equation}
\Phi \left( \bigcup\limits_{n\in \omega }U_{n}\right) \subseteq
\bigcup\limits_{n\in \omega }\Phi \left( U_{n}\right) \text{\label{Equation:
countable subadditivity}}
\end{equation}%
for any sequence $\left( U_{n}\right) _{n\in \omega }$ in $\mathcal{U}$. For 
$U\in \mathcal{U}$, set $m(U)=\lambda (\Phi (U))$. Using \cite[Proposition
3.2.7]{paterson_groupoids_1999}, one can extend $m$ to a Borel measure on $%
G^{0}$ by setting 
\begin{equation*}
m(E)=\inf \left\{ m(U)\colon U\in \mathcal{U}\text{, }E\subseteq U\right\}
\end{equation*}%
for $E\in \mathcal{B}(G^{0})$. Extend $\Phi $ to a homomorphism from $%
\mathcal{B}(G^{0})$ to $\mathcal{B}_{\lambda }$, by setting 
\begin{equation*}
\Phi (E)=\bigwedge \left\{ \Phi (U)\colon U\in \mathcal{U}\text{, }%
U\supseteq E\right\} \text{.}
\end{equation*}%
(The infimum exists by completeness of $\mathcal{B}_{\lambda }$.)

\begin{lemma}
The map $\Phi\colon \mathcal{B}(G^0)\to \mathcal{B}_\lambda$ is a $\sigma $%
-complete Boolean algebra homomorphism.
\end{lemma}

\begin{proof}
We claim that given $E_{0}$ and $E_{1}$ in $\mathcal{B}(G^{0})$, we have $%
\Phi (E_{0}\cap E_{1})=\Phi (E_{0})\cap \Phi (E_{1})$. \newline
\indent To prove the claim, observe that if $U_{j}$ is an open set
containing $E_{j}$ for $j\in \{0,1\}$, then $\Phi (U_{0}\cap U_{1})=\Phi
(U_{0})\cap \Phi (U_{1})$, and thus $\Phi (E_{0}\cap E_{1})\subseteq \Phi
(E_{0})\cap \Phi (E_{1}).$ In order to prove that equality holds, it is
enough to show that given $\varepsilon >0$, we have 
\begin{equation*}
\lambda \left( \Phi (E_{0}\cap E_{1})\right) \geq \lambda \left( \Phi \left(
E_{0}\right) \cap \Phi \left( E_{1}\right) \right) -\varepsilon .
\end{equation*}%
Fix an open set $U$ containing $E_{0}\cap E_{1}$ such that $m(U)\leq
m(E_{0}\cap E_{1})+\varepsilon $. Let $V_{0}$ and $V_{1}$ be open sets
satisfying $E_{j}\setminus (E_{0}\cap E_{1})\subseteq V_{j}$ for $j=0,1$,
and $\mu (V_{j})\leq \mu (E_{j}\setminus (E_{0}\cap E_{1}))+\varepsilon $.
For $j=0,1$, set $U_{j}=U\cup V_{j}$. Then $U_{j}\supseteq E_{j}$ and%
\begin{align*}
\lambda (\Phi (E_{0}\cap E_{1}))& =m(E_{0}\cap E_{1})\geq m(U)-\varepsilon
\geq m(U_{0}\cap U_{1})-3\varepsilon \\
& =\lambda (\Phi (U_{0}\cap U_{1}))-3\varepsilon =\lambda (\Phi (U_{0})\cap
\Phi (U_{1}))-3\varepsilon \geq \lambda (\Phi (E_{0})\cap \Phi
(E_{1}))-3\varepsilon \text{.}
\end{align*}%
We have therefore shown that $\Phi (E_{0}\cap E_{1})=\Phi (E_{0})\cap \Phi
(E_{1})$, so the claim is proved. It remains to show that if $(E_{n})_{n\in
\omega }$ is a sequence of pairwise disjoint Borel subsets of $G^{0}$, then 
\begin{equation*}
\Phi \left( \bigcup\limits_{n\in \omega }E_{n}\right) =\bigcup\limits_{n\in
\omega }\Phi \left( E_{n}\right) \text{.}
\end{equation*}%
By Equation~\eqref{Equation: countable subadditivity}, the left-hand side is
contained in the right-hand side. On the other hand, we have 
\begin{equation*}
\lambda \left( \Phi \left( \bigcup\limits_{n\in \omega }E_{n}\right) \right)
=m\left( \bigcup\limits_{n\in \omega }E_{n}\right) =\sum\limits_{n\in \omega
}m(E_{n})=\sum\limits_{n\in \omega }\lambda (\Phi (E_{n}))=\lambda \left(
\bigcup\limits_{n\in \omega }\Phi (E_{n})\right) ,
\end{equation*}%
so we conclude that equality holds, and the proof is complete.
\end{proof}

By \cite[Theorem~15.9]{kechris_classical_1995}, there is a Borel function $%
q\colon Z\rightarrow G^{0}$ such that $\Phi (E)=q^{-1}(E)$ for every $E\in 
\mathcal{B}(G^{0})$. Moreover, the map $q$ is unique up to $\lambda $-almost
everywhere equality.

\subsection{Measure}

Define a Borel probability measure $\mu$ on $G^0$ by $\mu =q_{\ast
}(\lambda) $. Consider the disintegration $\lambda =\int \lambda _{x}\ d\mu
(x) $ of $\lambda $ with respect to $\mu $, the Borel Banach bundle $%
\mathcal{Z}=\bigsqcup\limits_{x\in G^{0}}L^{p}(\lambda_x)$, and identify $%
L^{p}(\lambda) $ with $L^{p}(\mu,\mathcal{Z})$ as in Theorem~\ref{thm: BBb
structure}.

For $A\in \Sigma $, denote by $\theta _{A}\colon A^{-1}A\rightarrow AA^{-1}$
the homomorphism defined by $\theta _{A}(x)=r(Ax)$ for $x\in A^{-1}A$. Since 
$\rho (A)$ is a spatial partial isometry with domain $\Phi (s(A))$ and range 
$\Phi (r(A))$, there are a Borel function $g_{A}\colon \Phi
(r(A))\rightarrow \mathbb{C}$ and a Borel isomorphism $\phi _{A}\colon \Phi
(s(A))\rightarrow \Phi (r(A))$ such that 
\begin{equation}
(\rho (A)\xi )_{z}=g_{A}(z)\xi (\phi _{A}^{-1}(z))\text{\label{Equation:
spatial partial isometry}}
\end{equation}%
for $z\in \Phi (r(A))$. We claim that $(q\circ \phi _{A})(z)=(\theta
_{A}\circ q)(z)$ for $\lambda $-almost every $z\in \Phi (r(A))$. By the
uniqueness assertion in \cite[Theorem~15.9]{kechris_classical_1995}, it is
enough to show that $(\theta _{A}\circ q\circ \phi _{A}^{-1})^{-1}(U)=\Phi
(U)$ for every $U\in E(\Sigma )$ with $U\subseteq r(A)$. We have 
\begin{equation*}
(\theta _{A}\circ q\circ \phi _{A}^{-1})^{-1}(U)=(\phi _{A}\circ q^{-1}\circ
\theta _{A}^{-1})(U)=\phi _{A}(\Phi (\theta _{A}^{-1}(U)))=\phi _{A}(\Phi
(A^{-1}UA)).
\end{equation*}%
Given $\xi \in L^{p}(\lambda |_{\Phi (r(A))})$, set $\eta =\xi \circ \phi
_{A}$. Then 
\begin{align*}
\Delta _{\chi _{\phi _{A}\left( \Phi \left( A^{-1}UA\right) \right) }}\xi &
=\left( \Delta _{\chi _{\Phi \left( A^{-1}UA\right) }}\eta \right) \circ
\phi _{A}^{-1}=\left( \rho \left( A^{-1}UA\right) \eta \right) \circ \phi
_{A}^{-1}=\left( \rho (A)^{-1}\rho (U)\rho (A)\eta \right) \circ \phi
_{A}^{-1} \\
& =\left( \rho (A)^{-1}\rho (U)g_{A}\xi \right) \circ \phi _{A}^{-1}=\left(
\rho (A)^{-1}\chi _{\Phi (U)}g_{A}\xi \right) \circ \phi _{A}^{-1} \\
& =\left( \left( g_{A}\circ \phi _{A}\right) ^{-1}\left( \chi _{\Phi
(U)}\circ \phi _{A}\right) \left( g_{A}\circ \phi _{A}\right) \eta \right)
\circ \phi _{A}^{-1}=\chi _{\Phi (U)}\xi =\Delta _{\chi _{\Phi (U)}}\xi 
\text{.}
\end{align*}%
Thus $\Phi (U)=\phi _{A}(\Phi (A^{-1}UA))=(\theta _{A}\circ q\circ \phi
_{A}^{-1})^{-1}(U)$, and hence $(\theta _{A}\circ q\circ \phi
_{A}^{-1})(z)=q(z)$ for $\lambda $-almost every $z\in \Phi (r(A))$, as
desired. The claim is proved. It is shown in \cite[Proposition 3.2.2]%
{paterson_groupoids_1999} that $\mu $ is quasi-invariant whenever $(\theta
_{A})_{\ast }\mu |_{s(A)}\sim \mu |_{r(A)}$ for every open bisection $A$ of $%
G$. The same proof in fact shows that it is sufficient to check this
condition for $A\in \Sigma $. Given $A\in \Sigma $, we have 
\begin{eqnarray*}
\mu |_{r(A)} &=&q_{\ast }\lambda |_{\Phi (r(A))}\sim q_{\ast }((\phi
_{A})_{\ast }\lambda |_{\Phi (s(A))})=(q\circ \phi _{A})_{\ast }\lambda
|_{\Phi (s(A))}=(\theta _{A}\circ q)_{\ast }\lambda |_{\Phi (s(A))} \\
&=&(\theta _{A})_{\ast }(q_{\ast }\lambda |_{\Phi (s(A))})=(\theta
_{A})_{\ast }\mu |_{s(A)}\text{,}
\end{eqnarray*}%
so $\mu $ is quasi-invariant.

\subsection{Disintegration}

For $x\in G^{0}$, set $Z_{x}=q^{-1}(\{x\})$, and note that $Z_{x}=\Phi
(\{x\})$. Given $A\in \Sigma $, regard $\rho (A)$ as a surjective linear
isometry 
\begin{equation*}
\rho (A)\colon L^{p}(\lambda |_{\Phi (s(A))})\rightarrow L^{p}(\lambda
|_{\Phi (r(A))}).
\end{equation*}%
Let $\mathcal{Z}$ denote the Borel Banach bundle $\bigsqcup\nolimits_{x\in
G^{0}}L^{p}(\lambda _{x})$, and identify $L^{p}(\lambda |_{\Phi (s(A))})$
and $L^{p}(\lambda |_{\Phi (r(A))})$ with $L^{p}(\mu |_{s(A)},\mathcal{Z}%
|_{s(A)})$ and $L^{p}(\mu |_{r(A)},\mathcal{Z}|_{r(A)})$, respectively. If $%
U\in E(\Sigma )$ satisfies $U\subseteq r(A)$, one uses $\rho (A^{-1}UA)=\rho
(A)^{-1}\rho (U)\rho (A)$ to show that 
\begin{equation*}
\Delta _{U}\circ \rho (A)=\rho (A)\circ \Delta _{\theta _{A}(U)}\text{.}
\end{equation*}%
By Theorem~\ref{Theorem: Guichardet}, there is a Borel section $x\mapsto
T_{x}^{A}$ of $B(\mathcal{Z}|_{s(A)},\mathcal{Z}|_{r(A)},\theta _{A})$
consisting of invertible isometries, such that 
\begin{equation*}
(\rho (A)\xi )|_{Z_{y}}=\left( \frac{d(\theta _{A})_{\ast }\mu }{\ d\mu }%
(y)\right) ^{\frac{1}{p}}T_{\theta _{A}^{-1}(y)}^{A}\xi |_{Z_{\phi ^{-1}(y)}}
\end{equation*}%
for $\mu $-almost every $y\in r(A)$. Since 
\begin{equation*}
\left( \rho (A)\xi \right) |_{Z_{y}}=\left( g_{A}\right) |_{Z_{y}}\cdot
\left( \xi |_{Z_{y}}\circ \left( (\phi _{A})|_{Z_{\theta
_{A}^{-1}(y)}}^{|Z_{y}}\right) ^{-1}\right)
\end{equation*}%
for $\mu $-almost every $y\in r(A)$, we have 
\begin{equation*}
T_{x}^{A}\xi =\left( \frac{d(\theta _{A})_{\ast }\mu }{\ d\mu }\left( \theta
_{A}(x)\right) \right) ^{\frac{1}{p}}\left( g_{A}\right) |_{Z_{\theta
_{A}(x)}}\left( \xi \circ \left( (\phi _{A})|_{Z_{x}}^{|Z_{\theta
_{A}(x)}}\right) ^{-1}\right)
\end{equation*}%
for $\mu $-almost every $x\in s(A)$. Fix $A,B\in \Sigma $. Since $\rho $ is
a representation, $\rho (AB)=\rho (A)\rho (B)$. Therefore it follows from
the uniqueness of the direct integral representation of a decomposable
operator---see Remark \ref{Remark:uniqueness}---that $T_{x}^{AB}=T_{\theta
_{B}(x)}^{A}T_{x}^{B}$ for $\mu $-a.e. $x\in s(AB)$. Similarly if $A,B\in
\Sigma $ and $U\in E(\Sigma )$ is such that $AU=BU$, then the uniqueness of
the direct integral representation of a decomposable operator shows that $%
T_{x}^{A}=T_{x}^{B}$ for $\mu $-a.e. $x\in U$. Since $E(\Sigma )$ is a basis
for the topology of $G^{0}$ we conclude that $T_{x}^{A}=T_{x}^{B}$ for $\mu $%
-a.e. $x\in s(A)\cap s(B)$ such that $Ax=Bx$. A similar argument shows that,
if $A\in \Sigma $, then $(T_{x}^{A})^{-1}=T_{\theta _{A}(x)}^{A^{-1}}$ for $%
\mu $-a.e. $x\in s(A)$. It follows that, up to discarding a $\nu $-null set,
the assignment $T\colon G\rightarrow \mathrm{Iso}(\mathcal{Z})$ given by $%
T_{\gamma }=T_{s(\gamma )}^{A}$ for some $A\in \Sigma $ containing $\gamma $%
, determines a representation of $G$ on $\mathcal{Z}$. Indeed let $X$ be the
set of $x\in G^{0}$ such that

\begin{enumerate}
\item for every $A,B\in \Sigma $ such that $x\in s(AB)$, $%
T_{x}^{AB}=T_{\theta _{B}(x)}^{A}T_{x}^{B}$,

\item for every $A\in \Sigma $ such that $x\in s(A)$, $(T_{x}^{A})^{-1}=T_{%
\theta _{A}(x)}^{A^{-1}}$,

\item for every $A,B\in \Sigma $ such that $x\in s(A)\cap s(B)$ and $Ax=Bx$, 
$T_{x}^{A}=T_{x}^{B}$.
\end{enumerate}

Then by the discussion above and since $\Sigma $ is countable, $X$ is a $\mu 
$-conull subset of $G^{0}$. We claim that the restriction of $T$ to $G|_{X}$
is a groupoid homomorphism. Indeed if $\gamma ,\rho $ are elements of $%
G|_{X} $ such that $r(\rho )=s(\gamma )$ and $B,A\in \Sigma $ are such that $%
\gamma \in A$ and $\rho \in B$, then, since $s(\rho ),s(\gamma )\in X$, by
(1) and (3) we get that $T_{\gamma }T_{\rho }=T_{s(\gamma )}^{A}T_{s(\rho
)}^{B}=T_{s(\rho )}^{AB}=T_{\gamma \rho }$. Similarly applying (2) and (3)
one obtains that $T_{\gamma }^{-1}=T_{\gamma ^{-1}}$ for any $\gamma \in
G|_{X}$. This concludes the proof that the restriction of $T$ to $G|_{X}$ is
a groupoid homomorphism. If $A\in \Sigma $ then the maps $\gamma \mapsto
T_{s(\gamma )}^{A}$ and $\gamma \mapsto T_{\gamma }$ agree on $A\cap G|_{X}$%
. Since $x\mapsto T_{x}^{A}$ is a Borel map on $s(A)$ and $\Sigma $ is a
countable basis for the topology of $G$, it follows that the function $%
\gamma \mapsto T_{\gamma }$ is Borel on $G|_{X}$. This concludes the proof
that $T$ is a representation of $G$ on the Banach bundle $\mathcal{Z}$ in
the sense of Definition \ref{Definition: representation on Banach bundle}.
It is a consequence of Equation~\eqref{Equation: spatial partial isometry}
that%
\begin{equation*}
\left\langle \rho (A)\xi ,\eta \right\rangle =\int D\left( xA\right) ^{-%
\frac{1}{p}}\left\langle T_{xA}\xi _{\theta _{A}^{-1}(x)},\eta
_{x}\right\rangle \ d\mu (x)\text{,}
\end{equation*}%
for every $\xi \in L^{p}(\mathcal{Z},\mu )$ and every $\eta \in L^{p^{\prime
}}(\mu ,\mathcal{Z}^{\prime })$. This concludes the proof of Theorem~\ref%
{Theorem: disintegration representation inverse}.

\subsection{Correspondence}

Let as above $G$ be an étale groupoid, and $\Sigma _{c}\left( G\right) $ be
the inverse semigroup of precompact open bisections of $G$. Let $\Sigma $ be
an inverse subsemigroup of $\Sigma _{c}\left( G\right) $ that forms a basis
for the topology of $G$.

Suppose that $\pi $ is a contractive representation of $C_{c}(G)$ on $%
L^{p}(\lambda )$. It is shown in Subsection \ref{Subsection: integrated
representation slices} that, identifying $L^{p}\left( \lambda \right) $ with 
$L^{p}\left( \mu ,\mathcal{Z}\right) $, $\pi $ induces a tight, regular
representation of $\Sigma _{c}\left( G\right) $ on $L^{p}\left( \lambda
\right) $.\ Since $\Sigma $ forms a basis for the topology of $G$, the
restriction $\rho _{\pi }$ of such a representation to $\Sigma $ is still
tight by Lemma \ref{Lemma: dense}.

Let now $(T,\mu )$ be a representation of $G$ on an $L^{p}$-bundle $%
\bigsqcup\nolimits_{x\in G^{0}}L^{p}(\lambda _{x})$. Setting $\lambda :=\int
\lambda _{x}d\mu \left( x\right) $, one can identify $L^{p}\left( \mu ,%
\mathcal{Z}\right) $ with $L^{p}\left( \lambda \right) $ by Theorem \ref%
{thm: BBb structure}. Then one can consider the integrated form $\pi _{T}$
of $T$, which is a representation of $C_{c}\left( G\right) $ on $L^{p}\left(
\lambda \right) $. Then we set, following the notation above, $\rho
_{T}:=\rho _{\pi _{T}}$, which is a tight, regular representation of $\Sigma 
$ on $L^{p}\left( \lambda \right) $. An inspection of the definition of $%
\rho _{\pi }$ from Subsection \ref{Subsection: integrated representation
slices} shows that one can explicitly define $\rho _{T}$ via the formula 
\begin{equation*}
\langle \rho _{T}(A)\xi ,\eta \rangle =\int_{r(A)}D^{-\frac{1}{p}%
}(xA)\left\langle T_{xA}\xi _{\theta _{A}^{-1}(x)},\eta _{x}\right\rangle \
d\mu (x)
\end{equation*}%
for all $A\in \Sigma $, for all $\xi \in L^{p}(\mu ,\mathcal{Z})$, and all $%
\eta \in L^{p^{\prime }}(\mu ,\mathcal{Z}^{\prime })$.

\begin{theorem}
\label{Theorem: correspondence representations} Adopt the notation of the
comments above.

\begin{enumerate}
\item The assignment $T\mapsto \rho_T$ determines a bijective correspondence
between representations of $G$ on $L^p$-bundles and tight regular
representations of $\Sigma $ on $L^p$-spaces.

\item The assignment $\pi \mapsto \rho _{\pi }$ determines a bijective
correspondence between contractive representations of $C_{c}(G)$ on $L^{p}$%
-spaces and tight regular representations of $\Sigma $ on $L^{p}$ spaces.

\item The assignment $T\mapsto \pi _{T}$ is a bijective correspondence
between representations of $G$ on $L^{p}$-bundles and contractive
representations of $C_{c}(G)$ on $L^{p}$-spaces.
\end{enumerate}

Moreover, the correspondences in (1), (2), and (3) preserve the natural
relations of equivalence of representations.
\end{theorem}

\begin{proof}
First we show that, given a representation $T$ of $G$ on an $L^{p}$-bundle $%
\mathcal{Z}$, the corresponding representation $\rho _{T}$ of $\Sigma $ is
right. Consider the repr...

(1). This is an immediate consequence of the Disintegration Theorem~\ref%
{Theorem: disintegration representation inverse}.

(2). Suppose that $\rho $ is a tight representation of $\Sigma $ on $%
L^{p}(\lambda )$. Applying the Disintegration Theorem~\ref{Theorem:
disintegration representation inverse} one obtains a representation $\left(
\mu ,T\right) $ of $G$ on the bundle $\bigsqcup_{x\in G^{0}}L^{p}(\lambda
_{x})$ for a disintegration $\lambda =\int \lambda _{x}d\mu \left( x\right) $%
. One can then assign to $\rho $ the integrated form $\pi _{\rho }$ of $%
\left( \mu ,T\right) $. It is easy to verify that the maps $\rho \mapsto \pi
_{\rho }$ and $\pi \mapsto \rho _{\pi }$ are mutually inverse.

Finally (3) follows from combining (1) and (2).
\end{proof}

\section{\texorpdfstring{$L^p$}{Lp}-operator algebras of étale groupoids}

Throughout this section, we fix a H\"older exponent $p\in (1,\infty)$.

\subsection{$L^p$-operator algebras\label{Subsection: Lp operator algebras}}

\begin{definition}
A \emph{concrete }$L^{p}$-\emph{operator algebra} is a subalgebra $A$ of $%
B(L^{p}(\lambda ))$ for some $\sigma $-finite Borel measure $\lambda $ on a
standard Borel space. The identification of $M_{n}(A)$ with a subalgebra of $%
B(L^{p}(\lambda ^{(n)}))$ induces a norm on $M_{n}(A)$. The collection of
such norms defines a $p$-operator space structure on $A$ as in \cite[Section
4.1]{daws_p-operator_2010}. Moreover the multiplication on $A$ is a $p$%
-completely contractive bilinear map. Equivalently $M_{n}(A)$ is a Banach
algebra for every $n\in \mathbb{N}$.

An \emph{abstract $L^{p}$-operator algebra} is a Banach algebra $A$ endowed
with a $p$-operator space structure, which is $p$-completely isometrically
isomorphic to a concrete $L^{p}$-operator algebra.
\end{definition}

Let $A$ be a separable matricially normed algebra and let $\mathcal{R}$ be a
collection of $p$-completely contractive nondegenerate representations of $A$
on $L^{p}$-spaces. Set $I_{\mathcal{R}}=\bigcap\limits_{\pi \in \mathcal{R}}%
\mathrm{Ker}(\pi )$. Then $I_{\mathcal{R}}$ is an ideal in $A$. Arguing as
in \cite[Section 1.2.16]{blecher_operator_2004}, the completion $F^{\mathcal{%
R}}(A)$ of $A/I_{\mathcal{R}}$ with respect to the norm%
\begin{equation*}
\| a+I_{\mathcal{R}}\| =\sup \{\| \pi (a)\| \colon \pi \in \mathcal{R}\}
\end{equation*}%
for $a\in A$, is a Banach algebra. Moreover, $F^{\mathcal{R}}(A)$ has a
natural $p$-operator space structure that makes it into an (abstract) $L^{p}$%
-operator algebra.

\begin{remark}
If $\mathcal{R}$ separates the points of $A$, then the ideal $I_{\mathcal{R}%
} $ is trivial, and hence the canonical map $A\to F^{\mathcal{R}}(A) $ is an
injective $p$-completely contractive homomorphism.
\end{remark}

\begin{definition}
Let $\mathcal{R}^{p}$ be the collection of \emph{all }$p$-completely
contractive nondegenerate representations of $A$ on $L^{p}$-spaces
associated with $\sigma $-finite Borel measures on standard Borel spaces.
Then $F^{\mathcal{R}^{p}}(A)$ is abbreviated to $F^{p}(A)$, and called the 
\emph{enveloping $L^{p}$-operator algebra} of $A$.
\end{definition}

Suppose further that $A$ is a matricially normed *-algebra with a completely
isometric \emph{linear} involution $a\mapsto \check{a}$. (For example, for
an \'etale groupoid $G$, we endow $C_c(G)$ with its $I$-norm and the linear
involution $f\mapsto \check{f}$ defined before Definition~\ref{Definition:
dual representation}.) If $\pi \colon A\rightarrow B(L^{p}(\lambda ))$ is a $%
p$-completely contractive nondegenerate representation as before, then the
dual representation of $\pi $ is the $p^{\prime }$-completely contractive
nondegenerate representation $\pi ^{\prime }$ given by $\pi ^{\prime
}(a)=\pi (a^{\ast })^{\prime }$ for all $a\in A$.\newline
\indent Let $\mathcal{R}$ be a collection of $p$-completely contractive
nondegenerate representations of $A$ on $L^{p}$-spaces, and denote by $%
\mathcal{R}^{\prime }$ the collection of duals of elements of $\mathcal{R}$.
It is immediate that the involution of $A$ extends to a $p$-completely
isometric anti-isomorphism $F^{\mathcal{R}}(A)\rightarrow F^{\mathcal{R}%
^{\prime }}(A)$. Finally, since $(\mathcal{R}^{p})^{\prime }=\mathcal{R}%
^{p^{\prime }}$, the discussion above shows that the involution of $A$
extends to a $p$-completely isometric anti-isomorphism $F^{p}(A)\rightarrow
F^{p^{\prime }}(A)$.

\subsection{The full $L^p$-operator algebra of an étale groupoid}

Let $G$ be an étale groupoid. Recall that we can regard $C_{c}(G)$ as a
matricially normed *-algebra where $M_{n}(C_{c}(G))$ is endowed with the $I$%
-norm described in Subsection \ref{Subsection: Amplification of
representations}.

\begin{definition}
\label{Definition:full}We define the \emph{full }$L^{p}$-\emph{operator
algebra }$F^{p}(G)$ of $G$ to be the enveloping $L^{p}$-operator algebra of
the matricially normed *-algebra $C_{c}(G)$.
\end{definition}

\begin{remark}
By Proposition \ref{Proposition: separates points}, when $G$ is Hausdorff
the family of $p$-completely contractive nondegenerate representations of $%
C_{c}(G)$ on $L^{p}$-spaces separates the points of $C_{c}(G)$, and hence
the canonical map $C_{c}(G)\rightarrow F^{p}(G)$ is injective.
\end{remark}

The proof of the following is straightforward, and is left to the reader.

\begin{proposition}
\label{Proposition: correspondence Cc and Fp} The correspondence sending a $%
p $-completely contractive representation of $F^{p}(G)$ on an $L^{p}$-space
to its restriction to $C_{c}(G)$, is a bijective correspondence between $p$%
-completely contractive representations of $F^{p}(G)$ on $L^{p}$-spaces and $%
p$-completely contractive representations of $C_{c}(G)$ on $L^{p}$-spaces.
\end{proposition}

\begin{definition}
Let $\Sigma $ be an inverse semigroup, and consider the matricially normed
*-algebra structure on $\mathbb{C}\Sigma $ described in Definition \ref%
{Definition:CSIGMA}. Denote by $\mathcal{R}_{\mathrm{tight}}^{p}$ the
collection of tight representations of $\Sigma $ on $L^{p}$-spaces. We
define the \emph{tight enveloping $L^{p}$-operator algebra} of $\Sigma $,
denoted $F_{\mathrm{tight}}^{p}(\Sigma )$, to be $F^{\mathcal{R}_{\mathrm{%
tight}}^{p}}(\mathbb{C}\Sigma )$.
\end{definition}

\begin{remark}
Since the dual of a tight representation is also tight, it follows that the
involution on $\mathbb{C}\Sigma $ extends to a $p$-completely isometric
anti-isomorphism $F_{\mathrm{tight}}^{p}(\Sigma )\rightarrow F_{\mathrm{tight%
}}^{p^{\prime }}(\Sigma )$.
\end{remark}

From Theorem \ref{Theorem: correspondence representations} we can deduce the
following corollary.

\begin{corollary}
\label{Corollary: automatically completely}If $A$ is an $L^{p}$-operator
algebra, then any contractive homomorphism from $C_{c}(G)$ or $F^{p}(G)$ to $%
A$ is automatically $p$-completely contractive.
\end{corollary}

\begin{proof}
It is enough to show that any contractive nondegenerate representation of $%
C_{c}(G)$ on an $L^{p}$-space is $p$-completely contractive. This follows
from part (3) of Theorem~\ref{Theorem: correspondence representations},
together with the fact that the integrated form of a representation of $G$
on an $L^{p}$-bundle is $p$-completely contractive, as observed in
Subsection \ref{Subsection: Amplification of representations}.
\end{proof}

\begin{corollary}
\label{cor: tight inv smgp and gpid alg} Adopt the assumptions of Theorem~%
\ref{Theorem: correspondence representations}, and suppose moreover that $G$
is ample. Then the $L^{p}$-operator algebras $F^{p}(G)$ and $F_{\mathrm{tight%
}}^{p}(\Sigma )$ are $p$-completely isometrically isomorphic. In particular, 
$F^{p}(G)$ is generated by its spatial partial isometries.
\end{corollary}

\begin{proof}
Observe that when $G$ is ample, and $\Sigma $ is the inverse semigroup of
compact open slices, any tight representation of $\Sigma $ on an $L^{p}$%
-space is automatically regular. Thus the statement follows from part (2) of
Theorem~\ref{Theorem: correspondence representations}.
\end{proof}

\subsection{Reduced $L^{p}$-operator algebras of étale groupoids\label%
{Subsection: reduced}}

Let $\mu $ be a (not necessarily quasi-invariant) Borel probability measure
on $G^{0}$, and let $\nu $ be the measure on $G$ associated with $\mu $ as
in Subsection \ref{Subsection: background on groupoids}. Denote by $\mathrm{%
Ind}(\mu )\colon C_{c}(G)\rightarrow B(L^{p}(\nu ^{-1}))$ the left action by
convolution. Then $\mathrm{Ind}(\mu )$ is contractive and nondegenerate.

\begin{remark}
When $\mu $ is quasi-invariant, the representation $\mathrm{Ind}(\mu )$ is
the integrated form of the left regular representation $T^{\mu }$ of $G$ on $%
\bigsqcup\nolimits_{x\in G^{0}}\ell ^{p}(xG)$ as defined in Subsection \ref%
{Subsection: representations groupoids Lp-bundles}. When $G$ is Hausdorff,
the same argument as in Lemma~\ref{Lemma: faithful measure} shows that a
function $f$ in $C_{c}(G)$ belongs to $\mathrm{Ker}(\mathrm{Ind}(\mu ))$ if
and only if $f$ vanishes on the support of $\nu $.
\end{remark}

\begin{definition}
Define $\mathcal{R}_{\mathrm{red}}^{p}$ red to be the collection of
representations $\mathrm{\mathrm{Ind}}(\mu )$ where $\mu $ varies among the
Borel probability measures on $G^{0}$. The \emph{reduced }$L^{p}$-\emph{%
operator algebra } $F_{\mathrm{red}}^{p}(G)$ of $G$ is the enveloping $L^{p}$%
-operator algebra $F^{\mathcal{R}_{\mathrm{red}}^{p}}(C_{c}(G))$. The norm
on $F_{\mathrm{red}}^{p}(G)$ is denoted by $\| \cdot \| _{\mathrm{red}} $.
\end{definition}

The identity map on $C_{c}(G)$ induces a canonical $p$-completely
contractive homomorphism $F^{p}(G)\rightarrow F_{\mathrm{red}}^{p}(G)$ with
dense range. By Proposition \ref{Proposition: separates points}, when $G$ is
Hausdorff the family $\mathcal{R}_{\mathrm{red}}^{p}$ separates points, and
hence the canonical map $C_{c}(G)\rightarrow F_{\mathrm{red}}^{p}(G)$ is
injective.

\begin{remark}
The dual of $\mathrm{Ind}(\mu )\colon C_{c}(G)\rightarrow B(L^{p}(\nu
^{-1})) $ is the representation $\mathrm{Ind}(\mu )\colon
C_{c}(G)\rightarrow B(L^{p^{\prime }}(\nu ))$, and thus the involution on $%
C_{c}(G)$ extends to a $p$-completely isometric anti-isomorphism $F_{\mathrm{%
red}}^{p}(G)\rightarrow F_{\mathrm{red}}^{p^{\prime }}(G)$.
\end{remark}

For $x\in G^{0}$, we write $\delta _{x}$ for its associated point mass
measure, and write $\mathrm{Ind}(x)$ in place of $\mathrm{Ind}(\delta _{x})$%
. In this case, $\nu $ is the counting measure $c_{xG}$ on $xG$, and $\nu
^{-1}$ is the counting measure $c_{Gx}$ on $Gx$. Moreover, $\mathrm{Ind}(x)$
is given by 
\begin{equation*}
(\mathrm{Ind}(x)f(\xi ))(\rho )=\sum\limits_{\gamma \in r(\rho )G}f(\gamma
)\xi (\gamma ^{-1}\rho )
\end{equation*}%
for $f\in C_{c}(G)$, $\xi \in L^{p}(\nu ^{-1})$, and $\rho \in Gx$. The same
argument in the proof of \cite[Proposition 3.1.2]{paterson_groupoids_1999}
gives the following.

\begin{proposition}
\label{Proposition: induced representations} Let $\mu $ be a probability
measure on $G^{0}$. If $f\in C_{c}(G)$, then 
\begin{equation*}
\| \mathrm{Ind}(\mu )f\| =\sup_{x\in \mathrm{supp}(\mu )}\| \mathrm{Ind}%
(x)(f)\| \text{.}
\end{equation*}
\end{proposition}

\begin{corollary}
The algebra $F_{\mathrm{red}}^{p}(G)$ of $G$ is $p$-completely isometrically
isomorphic to the enveloping $L^{p}$-operator algebra $F^{\mathcal{R}%
}(C_{c}(G))$ with respect to the family of representations $\mathcal{R}=\{%
\mathrm{Ind}(x)\colon x\in G^{0}\}$.
\end{corollary}

\subsection{Amenable groupoids and their $L^{p}$-operator algebras}

There are several equivalent characterizations of amenability for étale
groupoids. By \cite[Theorem~2.2.13]{anantharaman-delaroche_amenable_2000},
an étale groupoid is amenable if and only if has an approximate invariant
mean; see \cite[Definition 4.1.1]{renault_c*-algebras_2009}

\begin{lemma}
\label{Lemma: amenable amplification} If $G$ is amenable and $m\geq 1$, then
its amplification $G_{m}$ is amenable.
\end{lemma}

\begin{proof}
Let $( f_{n}) _{n\in \omega }$ be an approximate invariant mean for $G$. For 
$n\in\omega$, define $f^{(m)}_n\colon C_c(G_m)\to\mathbb{C}$ by $f_{n}^{(m)
}( i,\gamma ,j) =\frac{1}{m}f_{n}(\gamma)$ for $(i,\gamma,j)\in G_m$. It is
not difficult to verify that $( f_{n}^{(m) }) _{n\in \omega }$ is an
approximate invariant mean for $G_m$. We omit the details.
\end{proof}

\begin{definition}
A pair of sequences $(g_n) _{n\in \omega }$ and $( h_{n})_{n\in \omega }$ of
positive functions in $C_c(G) $ is said to be an \emph{approximate invariant 
}$p$\emph{-mean} for $G$, if they satisfy the following:

\begin{enumerate}
\item $\sum\nolimits_{\gamma \in xG}g_{n}(\gamma )^{p}\leq 1$ and $%
\sum\nolimits_{\gamma \in xG}h_{n}(\gamma )^{p^{\prime }}\leq 1$ for every $%
n\in \omega $ and every $x\in G^{0}$,

\item the sequence of functions $x\mapsto \sum\nolimits_{\rho \in
xG}g_{n}(\rho )h_{n}(\rho )$ converges to $1$ uniformly on compact subsets
of $G^{0}$, and

\item the sequences of functions 
\begin{equation*}
\gamma \mapsto \sum\limits_{\rho \in r(\gamma )G}|g_{n}(\gamma ^{-1}\rho
)-g_{n}(\rho )|^{p}\quad \text{and}\quad \gamma \mapsto \ \sum\limits_{\rho
\in r(\gamma )G}|h_{n}(\gamma ^{-1}\rho )-h_{n}(\rho )|^{p^{\prime }}
\end{equation*}%
converge to $0$ uniformly on compact subsets of $G$.
\end{enumerate}
\end{definition}

It is not difficult to see that any amenable groupoid has an approximate
invariant $p$-mean. Indeed, if $( f_{n}) _{n\in \omega }$ is any approximate
invariant mean on $G$, then the sequences $( f_{n}^{1/p}) _{n\in \omega }$
and $( f_{n}^{1/p^{\prime }}) _{n\in\omega }$ define an approximate
invariant $p$-mean on $G$.

\begin{remark}
It is easy to check that if $( g_{n},h_{n}) _{n\in \omega }$ is an
approximately invariant $p$-mean on $G$, then $( h_{n}\ast g_{n}) _{n\in
\omega }$ converges to $1$ uniformly on compact subsets of $G$.
\end{remark}

The following theorem asserts that full and reduced $L^{p}$-operator
algebras of amenable étale groupoids are canonically isometrically
isomorphic. The proof is the straightforward generalization of \cite[Theorem
4.2.1]{renault_c*-algebras_2009} where $2$ is replaced by $p$, and
approximate invariant means are replaced by approximate invariant $p$-means.

\begin{theorem}
\label{thm: amenable groupoid} Suppose that $G$ is amenable. Then the
canonical homomorphism $F^{p}(G)\rightarrow F_{\mathrm{red}}^{p}(G)$ is a $p$%
-completely isometric isomorphism.
\end{theorem}

\section{Examples: analogs of Cuntz algebras and AF-algebras}

Throughout this section, we let $p\in (1,\infty)$.

\subsection{The Cuntz $L^p$-operator algebras}

Fix $d\in \omega $ with $d\geq 2$. The following is \cite[Definition 1.1]%
{phillips_analogs_2012} and \cite[Definition 7.4 (2)]{phillips_analogs_2012}%
. Algebra representations of complex unital algebras are always assumed to
be unital.

\begin{definition}
\label{Definition: Leavitt} Define the \emph{Leavitt algebra} $L_{d}$ to be
the universal (complex) algebra with generators $s_{0},\ldots
,s_{d-1},s_{0}^{\ast },\ldots ,s_{d-1}^{\ast }$, subject to the relations $%
\sum\nolimits_{j\in d}s_{j}s_{j}^{\ast }=1$ and $s_{j}^{\ast }s_{k}=\delta
_{j,k}$ for $j,k\in d$.

If $\lambda $ is a $\sigma $-finite Borel measure on a standard Borel space,
a \emph{spatial representation} of $L_{d}$ on $L^{p}(\lambda )$ is an
algebra homomorphism $\rho \colon L_{d}\rightarrow B(L^{p}(\lambda ))$ such
that for $j\in d$, the operators $\rho (s_{j})$ and $\rho (s_{j}^{\ast })$
are mutually inverse spatial partial isometries, i.e.\ $\rho (s_{j}^{\ast
})=\rho (s_{j})^{\ast }$.
\end{definition}

It is a consequence of a fundamental result of Cuntz from \cite%
{cuntz_simple_1977} that any two *-representations of $L_{d}$ on a Hilbert
space induce the same norm on $L_{d}$. The corresponding completion is the
Cuntz C*-algebra $\mathcal{O}_{d}$. Cuntz's result was later generalized by
Phillips in \cite{phillips_analogs_2012} to spatial representations of $%
L_{d} $ on $L^{p}$-spaces. Theorem~8.7 of \cite{phillips_analogs_2012}
asserts that any two spatial $L^{p}$-representations of $L_{d}$ induce the
same norm on it. The corresponding completion is the Cuntz $L^{p}$-operator
algebra $\mathcal{O}_{d}^{p}$; see \cite[Definition 8.8]%
{phillips_analogs_2012}. We now to explain how one can realize $\mathcal{O}%
_{d}^{p}$ as a groupoid $L^{p} $-operator algebra. Denote by $d^{\omega }$
the space of infinite sequences of elements of $d$, endowed with the product
topology. (Recall that $d$ is identified with the set $\left\{ 0,1,\ldots
,d-1\right\} $ of its predecessors.) Denote by $d^{<\omega }$ the space of
(possibly empty) finite sequences of elements of $d$. The length of an
element $a$ of $d^{<\omega }$ is denoted by $\mathrm{lh}(a)$. For $a\in
d^{<\omega }$ and $x\in d^{\omega } $, define $a^{\smallfrown }x\in
d^{\omega }$ to be the concatenation of $a$ and $x$. For $a\in d^{<\omega }$%
, denote by $[a]$ the set of elements of $d^{\omega }$ having $a$ as initial
segment. Clearly $\{[a]\colon a\in d^{<\omega }\}$ is a clopen basis for $%
d^{\omega }$.

\begin{definition}
The \emph{Cuntz inverse semigroup} $\Sigma _{d}$ is the inverse semigroup
generated by a zero $0$, a unit $1$, and elements $s_{j}$ for $j\in d$,
satisfying $s_{j}^{\ast }s_{k}=0$ whenever $j\neq k$.
\end{definition}

Set $s_{\varnothing }=1$ and $s_{a}=s_{a_{0}}\cdots s_{a_{lh(a)-1}}\in
\Sigma _{d}$ for $a\in d^{<\omega }\setminus \left\{ \varnothing \right\} $.
Every element of $\Sigma _{d}$ can be written uniquely as $s_{a}s_{b}^{\ast
} $ for some $a,b\in d^{<\omega }$.

\begin{remark}
The nonzero idempotents $E(\Sigma _{d})$ of $\Sigma _{d}$ are precisely the
elements of the form $s_{a}s_{a}^{\ast }$ for $a\in d^{<\omega }$. Moreover,
the function $d^{<\omega }\cup \{0\}\rightarrow E(\Sigma )$ given by $%
a\mapsto s_{a}s_{a}^{\ast }$ and $0\mapsto 0$, is a semilattice map, where $%
d^{<\omega }$ has its (downward) tree ordering defined by $a\leq b$ if and
only if $b$ is an initial segment of $a$, and $0$ is a least element of $%
d^{<\omega }\cup \{0\}$.
\end{remark}

Observe that if $a,b\in d^{<\omega }$, then $ab=0$ if and only if $a(j)\neq
b(j)$ for some $j\in \min \{\mathrm{lh}(a),\mathrm{lh}(b)\}$

\begin{lemma}
\label{Lemma: tight representation E(Cuntz)} Let $\mathcal{B}$ be a Boolean
algebra and let $\beta \colon d^{<\omega }\rightarrow \mathcal{B}$ be a
representation. Then $\beta $ is tight if and only if $\beta (\varnothing
)=1 $ and $\beta (a)\leq \bigvee\nolimits_{j\in d}\beta (a^{\smallfrown }j)$%
for every $a\in d^{<\omega }$.
\end{lemma}

\begin{proof}
Suppose that $\beta $ is tight. Since $1$ is a cover of $E^{\varnothing
,\varnothing }$, we have $\beta (\varnothing )=1$. Similarly, $%
\{a^{\smallfrown }j\colon j\in d\}$ is a cover of $E^{\{a\},\varnothing }$
and thus $\beta (a)\leq \bigvee\nolimits_{j\in d}\beta (a^{\smallfrown }j)$.
Let us now show the \textquotedblleft if\textquotedblright\ implication. By 
\cite[Proposition 11.8]{exel_inverse_2008}, it is enough to show that for
every $a\in d^{<\omega }$ and every finite cover $Z$ of $\{a\}$, one has $%
\beta (a)\leq \bigvee\nolimits_{z\in Z}\beta (z)$. That this is true follows
from the hypotheses, using induction on the maximum length of elements of $Z$%
.
\end{proof}

\begin{lemma}
\label{Lemma: tight representation Cuntz} Let $\lambda $ be a $\sigma $%
-finite Borel measure on a standard Borel space, and $\rho $ be a
representation of $\Sigma _{d}$ on $L^{p}(\lambda )$. Then $\rho $ is tight
if and only if $\sum\nolimits_{j\in d}\rho (s_{j}s_{j}^{\ast })=\rho (1)=1.$
\end{lemma}

\begin{proof}
Suppose that $\rho $ is tight. Then $\rho |_{E(\Sigma )}$ is tight and
therefore%
\begin{equation*}
1=\rho (1)=\bigvee\limits_{j\in d}\rho (s_{j}s_{j}^{\ast
})=\sum\limits_{j\in d}\rho (s_{j}s_{j}^{\ast })
\end{equation*}%
by Lemma~\ref{Lemma: tight representation E(Cuntz)}. Conversely, given $a\in
d^{<\omega }$, we have 
\begin{eqnarray*}
\sum\limits_{j\in d}\rho (s_{a^{\smallfrown }j}s_{a^{\smallfrown }j}^{\ast
}) &=&\sum\limits_{j\in d}\rho (s_{a}s_{j}s_{j}^{\ast }s_{a}^{\ast
})=\sum\limits_{j\in d}\rho (s_{a})\rho (s_{j}s_{j}^{\ast })\rho
(s_{a}^{\ast }) \\
&=&\rho (s_{a})\left( \sum\limits_{j\in d}\rho (s_{j}s_{j}^{\ast })\right)
\rho (s_{a}^{\ast })=\rho (s_{a})\rho (1)\rho (s_{a}^{\ast })=\rho
(s_{a}s_{a}^{\ast })\text{,}
\end{eqnarray*}%
which shows that $\rho $ is tight, concluding the proof.
\end{proof}

\begin{proposition}
\label{Proposition: Cuntz algs 1} The algebra $F_{\mathrm{tight}}^{p}(\Sigma
_{d})$ is $p$-completely isometric isomorphic to $\mathcal{O}_{d}^{p}$.
\end{proposition}

\begin{proof}
Observe that the Leavitt algebra $L_{d}$ (see Definition \ref{Definition:
Leavitt}) is isomorphic to the quotient of $\mathbb{C}\Sigma _{d}$ by the
ideal generated by the elements $\delta _{1}-\sum\limits_{j\in d}\delta
_{s_{j}s_{j}^{\ast }}$ and $\delta _{0}$. (Here, $\delta _{s}$ denotes the
canonical element in $\mathbb{C}\Sigma _{d}$ corresponding to $s\in \Sigma
_{d}$.) By Lemma~\ref{Lemma: tight representation Cuntz}, tight
representations of $\Sigma _{d}$ correspond precisely to spatial
representations of the Leavitt algebra $L_{d}$ as defined in \cite[%
Definition 7.4]{phillips_analogs_2012}. The result then follows.
\end{proof}

It is well known that $\Sigma _{d}$ the inverse semigroup of compact open
bisections of the ample groupoid $\mathcal{G}_{d}$ described in \cite%
{renault_c*-algebras_2009} (and denoted by $\mathcal{O}_{d}$ therein).

\begin{theorem}
Let $d\geq 2$ be a positive integer, and let $\mathcal{G}_{d}$ denote the
corresponding Cuntz groupoid. Then $F^{p}(\mathcal{G}_{d})$ is canonically $%
p $-completely isometrically isomorphic to $\mathcal{O}_{d}^{p}$.
\end{theorem}

\begin{proof}
It is easy to check that the function $s_{a}s_{b}^{\ast }\mapsto \lbrack
a,b] $ defines an injective homomorphism from $\Sigma _{d}$ to the inverse
semigroup of compact open bisections of $\mathcal{G}_{d}$. It is well known
that $\mathcal{G}_{d}$ is amenable; see \cite[Exercise 4.1.7]%
{renault_c*-algebras_2009}. It follows from Theorem~\ref{thm: amenable
groupoid}, Corollary~\ref{cor: tight inv smgp and gpid alg}, and Proposition %
\ref{Proposition: Cuntz algs 1}, that there are canonical $p$-completely
isometric isomorphisms 
\begin{equation*}
F_{\mathrm{red}}^{p}(\mathcal{G}_{d})\cong F^{p}(\mathcal{G}_{d})\cong F_{%
\mathrm{tight}}^{p}(\Sigma _{d})\cong \mathcal{O}_{d}^{p}.\qedhere
\end{equation*}
\end{proof}

\subsection{Analogs of AF-algebras on $L^{p}$-spaces}

In this subsection, we show how one can use the machinery developed in the
previous sections to construct those $L^{p}$-analogs of AF-algebras that
look like C*-algebras, and which are called ``spatial" in \cite%
{phillips_analogs_2014}.

Fix $n\in \mathbb{N}$. The algebra $M_{n}( \mathbb{C}) $ of $n\times n$
matrices with complex coefficients can be (algebraically) identified with $%
B( \ell ^{p}( n) ) $. This identification turns $M_{n}( \mathbb{C}) $ into
an $L^{p}$-operator algebra that we will denote---consistently with \cite%
{phillips_analogs_2012}---by $M_{n}^{p}$. It is not difficult to verify that 
$M_{n}^{p}$ can be realized as a groupoid $L^{p}$-operator algebra, and we
proceed to outline the argument.

Denote by $T_{n}$ the principal groupoid determined by the trivial
equivalence relation on $n$. It is well-known (see \cite[page 121]%
{renault_groupoid_1980}) that $T_{n}$ is amenable. Moreover, the inverse
semigroup $\Sigma _{\mathcal{K}}(T_{n})$ of compact open bisections of $%
T_{n} $, is the inverse semigroup generated by a zero element $0$, a unit $1$%
, and elements $e_{jk}$ for $j,k\in n$, subject to the relations $%
e_{jk}^{\ast }e_{\ell m}=\delta _{k\ell }e_{jm}$ for $j,k,\ell ,m\in n$.
Since $\left\{ e_{jj}:j\in n\right\} $ for a cover of $1$ in the sense of
Definition \ref{Definition:cover}, a tight $L^{p}$-representation $\rho $ of 
$\Sigma _{\mathcal{K}}\left( T\right) $ satisfies $1=\rho (1)=\sum_{j\in
n}\rho (e_{jj})$. It thus follows from \cite[Theorem~7.2]%
{phillips_analogs_2012} that the map from $M_{n}^{p}$ to the range of $\rho $%
, defined by assigning $\rho \left( e_{jk}\right) $ to the $jk$-th matrix
unit in $M_{n}^{p}$, is isometric. We conclude that $F^{p}(T_{n})$ is
isometrically isomorphic to $M_{n}^{p}$. Reasoning in the same way at the
level of amplifications shows that $F^{p}(T_{n})$ and $M_{n}^{p}$ are in
fact $p$-completely isometrically isomorphic.

If $k\in \mathbb{N}$ and $\mathbf{n}=\left( n_{0},\ldots ,n_{k-1}\right) $
is a $k$-tuple of natural numbers, then the Banach algebra $%
M_{n_{0}}^{p}\oplus \cdots \oplus M_{n_{k-1}}^{p}$ acts naturally on the $%
L^{p}$-direct sum $\ell ^{p}(n_{0})\oplus _{p}\cdots \oplus _{p}\ell
^{p}(n_{k-1})\cong \ell ^{p}(n_{0}+\cdots +n_{k-1})$. The Banach algebra $%
M_{n_{0}}^{p}\oplus \cdots \oplus M_{n_{k-1}}^{p}$ can also be realized as
groupoid $L^{p}$-operator algebra by considering the disjoint union of the
groupoids $T_{n_{0}},T_{n_{1}},\ldots ,T_{n_{k-1}}$.

Here is the definition of spatial $L^p$-operator AF-algebras

\begin{definition}
\label{definition: spatial AF-algebra} A separable Banach algebra $A$ is
said to be a \emph{spatial $L^p$-operator AF-algebra} if there exists a
direct system $(A_n,\varphi_n)_{n\in\omega}$ of $L^p$-operator algebras $A_n$
which are isometrically isomorphic to algebras of the form $%
M_{n_0}^p\oplus\cdots\oplus M_{n_k}^p$, with isometric connecting maps $%
\varphi_n\colon A_n\to A_{n+1}$, and such that $A$ is isometrically
isomorphic to the direct limit $\varinjlim (A_n,\varphi_n)_{n\in\omega}$.
\end{definition}

Banach algebras as in the definition above, as well as more general direct
limits of semisimple finite-dimensional $L^{p}$-operator algebras, will be
studied in \cite{phillips_analogs_2014}.

In the rest of this subsection, we will show that spatial $L^{p}$-operator
AF-algebras can be realized as groupoid $L^{p}$-operator algebras. \newline

\subsubsection{Spatial $L^p$-operator UHF-algebras.}

For simplicity, we will start by observing that spatial $L^{p}$-operator
UHF-algebras are groupoid $L^{p}$-operator algebras. Spatial $L^{p}$%
-operator UHF-algebras are the spatial $L^{p}$-operator AF-algebras where
the building blocks $A_{n}$ appearing in the definition are all full matrix
algebras $M_{d_{n}}^{p}$ for some $d_{n}\in \omega $. These have been
defined and studied in \cite{phillips_isomorphism_2013}.

Let $d=(d_{n})_{n\in \omega }$ be a sequence of positive integers. Denote by 
$A_{d}^{p}$ the corresponding $L^{p}$-operator UHF-algebra defined as above;
see also \cite[Definition 3.9]{phillips_simplicity_2013}. In the following
we will show that $A_{d}^{p}$ is the enveloping algebra of a natural
groupoid associated with the sequence $d$. Define $Z_{d}=\prod_{j\in n}d_{j}$%
, and consider the groupoid%
\begin{equation*}
G_{d}=\left\{ \left( \alpha ^{\smallfrown }x,\beta ^{\smallfrown }x\right)
\colon \alpha ,\beta \in \prod_{j\in n}d_{j},x\in \prod_{j\geq n}d_{j},n\in
\omega \right\}
\end{equation*}%
having $Z_{d}$ as set of objects. (Here we identify $x\in Z_{d}$ with the
pair $(x,x)\in G_{d}$.) The operations are defined by $s(\alpha
^{\smallfrown }x,\beta ^{\smallfrown }x)=\beta ^{\smallfrown }x$, $(\alpha
^{\smallfrown }x,\beta ^{\smallfrown }x)^{-1}=(\beta ^{\smallfrown }x,\alpha
^{\smallfrown }x)$, and $(\alpha ^{\smallfrown }x,\beta ^{\smallfrown
}x)(\gamma ^{\smallfrown }y,\delta ^{\smallfrown }y)=(\alpha ^{\smallfrown
}x,\delta ^{\smallfrown }y)$ when $\beta ^{\smallfrown }x=\gamma
^{\smallfrown }y$. It is well-known that $G_{d}$ is amenable; see \cite[%
Section~4.2]{renault_c*-algebras_2009}, and specifically Theorem~4.2.5 there.

Given $k\in \omega $ and given $\alpha $ and $\beta $ in $\prod_{j\in
k}d_{j} $, define $U_{\alpha \beta }$ to be the set of $\left( \alpha
^{\smallfrown }x,\beta ^{\smallfrown }x\right) \in G_{d}$ for $x\in
\prod_{j\geq k}d_{j}$. Then the collection of $U_{\alpha ,\beta }$ for $%
\alpha ,\beta \in \prod\nolimits_{j\in k}d_{j}$ and $k\in \omega $ is a
basis of compact open bisections for an ample groupoid topology on $G_{d}$. 
\newline
\indent Fix $k\in \omega $ and consider the compact groupoid $G_{d}^{k}$
given by the union of $U_{\alpha ,\beta }$ for $\alpha ,\beta \in
\prod_{j\in k}d_{j}$. The groupoid $G_{d}$ can be seen as the topological
direct limit of the system $(G_{d}^{k})_{k\in \omega }$. It is clear that,
if $n=d_{0}\cdots d_{k-1}$, then $G_{d}^{k}$ is isomorphic to the groupoid $%
T_{n}$ defined previously. Therefore $F^{p}(G_{d}^{k})$ is isometrically
isomorphic to $M_{d_{0}\cdots d_{k-1}}^{p}$.

For $k\in \mathbb{N}$, identify $C(G_{d}^{k})$ with a *-subalgebra of $%
C_{c}(G_{d})$, by setting $f\in C(G_{d}^{k})$ to be $0$ outside $G_{d}^{k}$.
For $k<n$, we claim that the inclusion map from $C(G_{d}^{k})$ to $%
C(G_{d}^{n})$ induces an isometric embedding $\varphi _{n}\colon
F^{p}(G_{d}^{k})\rightarrow F^{p}(G_{d}^{n})$. This can be easily verified
by direct computation, after noticing that $G_{d}^{k}$ and $G_{d}^{n}$ are
amenable, and hence the full and reduced norms on $C(G_{d}^{k})$ and $%
C(G_{d}^{n})$ coincide. One then obtains a direct system $\left(
F^{p}(G_{d}^{k}),\varphi _{n}\right) _{n\in \mathbb{N}}$ with isometric
connecting maps whose limit is $F^{p}(G)$. Since $F^{p}(G_{d}^{k})\cong
M_{d_{0}\cdots d_{k-1}}^{p}$ as observed above, we conclude that $%
F^{p}(G_{d})\cong A_{d}^{p}$.

\subsubsection{Spatial $L^p$-operator AF-algebras.}

As in the C*-algebra case, there is a natural correspondence between $L^{p}$%
-operator AF-algebras and Bratteli diagrams. (For the definition of Bratteli
diagrams, see \cite[Subsection 7.2.3]{rordam_introduction_2000}.) Let $(E,V) 
$ be a Bratteli diagram, and $A^{(E,V) }$ be the associated $L^{p}$-operator
AF-algebra. The algebra $A^{(E,V)}$ can be defined in the same way that AF $%
C^*$-algebras are defined from a Bratteli diagram, as direct limits of sums
of finite dimensional $C^*$-algebras, except that each matrix algebra is now
given its (spatial) $L^p$-operator norm, as described at the beginning of
this subsection. With this definition of spatial $L^p$-operator AF-algebras,
it is unclear that such algebras are indeed $L^p$-operator algebras, since
it is not known whether $L^p$-operator algebras are closed under direct
limits (with contractive maps). In the following, we will explain how to
realize $A^{(E,V) }$ as a groupoid $L^{p}$-operator algebra. As a
consequence, it will follow that spatial $L^p$-operator AF-algebras are
always representable on $L^p$-spaces, which was not known before.

Denote by $X$ the set of all infinite paths in $(E,V)$. Then $X$ is a
compact zero-dimensional space. Denote by $G^{(E,V)}$ the tail equivalence
relation on $X$, regarded as a principal groupoid having $X$ as set of
objects. It is well known that $G^{(E,V)}$ is amenable; see \cite[Chapter
III, Remark 1.2]{renault_c*-algebras_2009}. If $\alpha ,\beta $ are \emph{%
finite} paths of the same length and with the same endpoints, define $%
U_{\alpha \beta }$ to be the set of elements of $G^{(E,V)}$ of the form $%
\left( \alpha ^{\smallfrown }x,\beta ^{\smallfrown }x\right) $. The
collection of all the sets $U_{\alpha \beta }$ is a basis for an ample
groupoid topology on $G^{(E,V)}$. For $k\in \omega $, let $G_{k}^{(E,V)}$ be
the union of $U_{\alpha \beta }$ over all finite paths $\alpha ,\beta $ as
before that moreover have length at most $k$. Then $G_{n}^{(E,V)}$ is a
compact groupoid and $G$ is the topological direct limit of $%
(G_{k}^{(E,V)})_{k\in \omega }$.\newline
\indent Fix $k\in \omega $. Denote by $l$ the cardinality of the $k$-th
vertex set $V_{k}$. Denote by $n_{0},\ldots ,n_{l-1}$ the \emph{%
multiplicities }of the vertices in $V_{k}$. (The multiplicity of a vertex in
a Bratteli diagram is defined in the usual way by recursion.) Set $\mathbf{n}%
=(n_{0},\ldots ,n_{l-1})$, and observe that $G_{k}^{(E,V)}$ is isomorphic to
the groupoid $T_{\mathbf{n}}$ as defined above. In particular $%
F^{p}(G_{n}^{(E,V)})\cong M_{n_{0}}^{p}\oplus \cdots \oplus M_{n_{l-1}}^{p}$%
. As before, one can show that the direct system $(F^{p}(G_{n}^{(E,V)}))_{n%
\in \omega }$ has isometric connecting maps, and that the inductive limit is 
$F^{p}(G^{(E,V)})$. This concludes the proof that $A^{(E,V)}$ is $p$%
-completely isometrically isomorphic to $F^{p}\left( G^{(E,V)}\right) $. In
particular, this shows that $A^{(E,V)}$ is indeed an $L^{p}$-operator
algebra.

\section{Concluding remarks and outlook}

It is not difficult to see that the class of $L^{p}$-operator algebras is
closed---within the class of all matricially normed Banach algebras---under
taking subalgebras and ultraproducts. As noted by Ilijas Farah and N.\
Christopher Phillips, this observation, together with a general result from
logic for metric structures, implies that the class of $L^{p}$-operator
algebras is---in model-theoretic jargon---\emph{universally axiomatizable}.
This means that $L^{p}$-operator algebras can be characterized as those
matricially normed Banach algebras satisfying certain expressions only
involving the algebra operations, the matrix norms, continuous functions
from $\mathbb{R}^{n}$ to $\mathbb{R}$, and suprema over balls of matrix
amplifications.

Determining what these expressions are seems to be, in our opinion, an
important problem in the theory of algebras of operators on $L^{p}$-spaces.

\begin{problem}
\label{Problem:axiomatize}Find an explicit intrinsic characterization of $%
L^{p}$-operator algebras within the class of matricially normed Banach
algebras.
\end{problem}

An explicit characterization of algebras acting on\emph{\ subspaces of
quotients }of $L^{p}$-spaces was provided by Le Merdy in \cite%
{merdy_representation_1996}.\ These are precisely the matricially normed
Banach algebras that are moreover $p$-operator spaces in the terminology of 
\cite{daws_p-operator_2010}, and such that multiplication is $p$-completely
contractive. Similar results have been obtained by Junge for algebras of
operators on subspaces of $L^{p}$-spaces; see \cite[Corollary~1.5.2.2]%
{junge_factorization_1996}.

A stumbling block towards answering Problem \ref{Problem:axiomatize} is the
fact that $L^{p}$-operator algebras are not closed under quotients; see \cite%
{gardella_quotients_2014}. This gives a lower bound on the complexity of a
possible axiomatization of $L^{p}$-operator algebras. Precisely, it implies
that $L^{p}$-operator algebras do not form a \emph{variety }of Banach
algebras in the sense of \cite{dixon_varieties_1976}, unlike algebras acting
on subspaces of quotients of $L^{p}$-spaces.

It is conceivable that, for sufficiently nice groupoids, ideals of the
associated $L^{p}$-operator algebras correspond to subgroupoid. In turn
quotients would correspond to quotients at the groupoid level. In particular
this would imply that such groupoid $L^{p}$-operator algebras are indeed
closed under quotients. This would be another feature of groupoid $L^{p}$%
-operator algebra, and similarity trait with C*-algebras. Such a general
result about quotients would also simplify the task of proving simplicity of 
$L^{p}$-operator algebras coming from groupoids. This problem has been dealt
by ad hoc methods for UHF and Cuntz $L^{p}$ operator algebras in \cite%
{phillips_simplicity_2013}.

\begin{problem}
Is $F_{red}^{p}(G)$ simple whenever $G$ is a minimal and topologically
principal étale groupoid?
\end{problem}

A potential application of groupoids to the theory of $L^{p}$-operator
algebras comes from the technique of Putnam subalgebras. Let $X$ be a
compact metric space and let $h\colon X\rightarrow X$ be a homeomorphism.
Denote by $u$ the canonical unitary in the C*-crossed product $C^{\ast }(%
\mathbb{Z},X,h)$ implementing $h$. If $Y$ is a closed subset of $X$, then
the corresponding \emph{Putnam subalgebra} $C^{\ast }(\mathbb{Z},X,h)_{Y}$
is the C*-subalgebra of $C^{\ast }(\mathbb{Z},X,h)$ generated by $C(X)$ and $%
uC_{0}(X\setminus Y)$. It is known that $C^{\ast }(\mathbb{Z},X,h)_{Y}$ can
be described as the enveloping C*-algebra of a suitable étale groupoid.

In the context of C*-algebras, Putnam subalgebras are fundamental in the
study of transformation group C*-algebras of minimal homeomorphisms. For
example, Putnam showed in \cite[Theorem~3.13]{putnam_c*-algebras_1989} that
if $h$ is a minimal homeomorphism of the Cantor space $X$, and $Y$ is a
nonempty clopen subset of $X$, then $C^{\ast }(\mathbb{Z},X,h)_{Y}$ is an
AF-algebra. This is then used in \cite{putnam_c*-algebras_1989} to prove
that the crossed product $C^{\ast} ( \mathbb{Z},X,h) _{Y}$ is a simple A$%
\mathbb{T} $-algebra of real rank zero. Similarly, Putnam subalgebras were
used by Huaxin Lin and Chris Phillips in \cite{lin_crossed_2010} to show
that, under a suitable assumption on $K$-theory, the crossed product of a
finite-dimensional compact metric space by a minimal homeomorphism is a
simple unital C*-algebra with tracial rank zero.

Considering the groupoid description of Putnam subalgebras provides a
natural application of our constructions to the theory of $L^{p}$-crossed
products introduced in \cite{phillips_crossed_2013}. It is conceivable that
with the aid of groupoid $L^{p}$-operator algebras, Putnam subalgebras could
be used to obtain generalizations of the above mentioned results to $L^{p}$%
-crossed products.




\providecommand{\bysame}{\leavevmode\hbox to3em{\hrulefill}\thinspace} %
\providecommand{\MR}{\relax\ifhmode\unskip\space\fi MR } 
\providecommand{\MRhref}[2]{  \href{http://www.ams.org/mathscinet-getitem?mr=#1}{#2}
} \providecommand{\href}[2]{#2}

\end{document}